\documentclass[reqno,11pt]{amsart}

\usepackage[a4paper,margin=2cm]{geometry}
\usepackage[utf8]{inputenc}
\usepackage{amsmath, amssymb, amsfonts,mathtools}
\allowdisplaybreaks

\usepackage{graphicx}
\usepackage{subcaption}
\usepackage{pgf,tikz,pgfplots}
\usetikzlibrary{intersections,shapes,calc,arrows,shadings}
\usepackage{xcolor}
\usepackage{array,ragged2e}

\usepackage{url}
\usepackage{algorithmicx}
\usepackage{algpseudocode}
\usepackage{float}
\floatstyle{boxed}
\newfloat{Algorithm}{}{prog}
\algrenewcommand{\algorithmicrequire}{\textbf{Input:}}
\algrenewcommand{\algorithmicensure}{\textbf{Output:}}
\algnewcommand{\AND}{\textbf{ and }}
\usepackage{setspace}
\usepackage{hyperref}

\definecolor{myred}{rgb}{0,0,0}
\def\thline{\noalign{\hrule height.8pt}}

\numberwithin{equation}{section}
\newtheorem{theorem}{Theorem}[section]

\newtheorem{lemma}{Lemma}[section]
\newtheorem{definition}{Definition}[section]

\newtheorem{remark}{Remark}[section]

\def\bc{\mathbf{c}}
\def\bn{\mathbf{n}}
\def\bt{\mathbf{t}}
\def\br{\mathbf{r}}
\def\btau{\mathbf{\tau}}
\long\def\sefue#1{}
\def\myline(#1,#2)(#3,#4){
	((#2 - #4)*(\t - #1)/(#1 - #3)) + #2
}

\newcommand*{\parambernbez}[4]{
	plot ({(array({#3},0)*(\t)^2 + array({#2},0)*2*#4*(1-\t)*\t + array({#1},0)*(1-\t)^2)/((\t)^2 + 2*#4*(1-\t)*\t + (1-\t)^2)},{(array({#3},1)*(\t)^2 + array({#2},1)*2*#4*(1-\t)*\t + array({#1},1)*(1-\t)^2)/((\t)^2 + 2*#4*(1-\t)*\t + (1-\t)^2)})
}

\title[\emph{ConicCurv}]{\emph {ConicCurv}: A curvature estimation algorithm for planar polygons}

\author[R. D\'iaz-Fuentes et al]{$^{\natural}$Rafael D\'iaz-Fuentes 
\and $^{\sharp}$Jorge Estrada-Sarlabous 
\and $^{\sharp}$Victoria Hern\'andez-Mederos}

\begin{document}
	
	\maketitle
	\centerline{$^{\natural}$Department of Mathematics and Computer Science, University of Cagliari, Cagliari, Italy} 
	\centerline{({\tt rafael.diazfuentes@unica.it}) }
	\medskip
	\centerline{$^{\sharp}$Instituto de Cibern\'etica, Matem\'atica y F\'isica, La Habana, Cuba} 
	\centerline{({\tt \{jestrada,vicky\}@icimaf.cu}).}
	
	\begin{abstract} \justifying 
		\emph{ConicCurv} is a new \emph{derivative-free} algorithm to estimate the curvature of a plane curve from a sample of data points. It is based on a known tangent estimator method grounded on classic results of Projective Geometry and B\'ezier rational conic curves. The curvature values estimated by \emph{ConicCurv} are invariant to Euclidean changes of coordinates and reproduce the exact curvature values if the data are sampled from a conic.\\
		We show that \emph {ConicCurv} has convergence order $3$ and, if the sample points are \emph {uniformly arc-length distributed}, the convergence order is $4$. The performance of \emph {ConicCurv} is compared with some of the most frequently used algorithms to estimate curvatures and its performance is illustrated in the calculation of the elastic energy of subdivision curves and the location of L-curves corners.
		
	\end{abstract}
		
	\section{Introduction}
	\label{sec:intro}
	
	In several tasks of Computer Aided Geometric Design and Computer Vision it is necessary to estimate curvature values from a sample of few unevenly distributed points \cite{Belyaev,Be99,Ca98,curvat3D,Lewiner,Wo93}. Those values may be used to design curves or surfaces for free design applications, such as typography design, cartoons, and games, among others. 
	The need comes from the interpolation or approximation of a planar set of points with methods that require the estimation of curvature values \cite{SofiaCubicAspline}.
	Another situation where the data consists of a set of few nonuniform distributed points is found in some Tykhonov regularization problems \cite{CMCR00,CGG02,KiRa19}, where the points describe a curve with L-shape or in another cases an U-shape \cite{ChenUcurve,Krawczyk}. In these problems, the generation of points on the curves is computationally expensive. When we seek to estimate the location of the points of relative maximum of curvature (corner points) of L-curves or U-curves, it is necessary to locate the desired point with a small sample of points.

	The most known curvature estimation methods in the literature use tangent vectors, osculating circles, or first and second order derivatives, depending on the selected curvature definition. In the first case one needs to estimate the derivative of the tangent vector with respect to arc-length. In the second case, the radius of the osculating circle that touches the curve at the desired point is estimated as the radius of the circumference that interpolates that point and two consecutive ones (anterior and posterior neighbors, according to the order of the data). This method offers an intuitive control of the assigned curvature values, but only reproduces the exact curvature values if the data come from a circle, which is a very particular case of the shape that the designer could wish for. In the third case, the estimated values of the first and second derivatives at the point are used to compute the curvature. We refer the interested reader to \cite{Belyaev,Hermann07,Lewiner}, and references therein, for a fairly complete study of methods to estimate curvature.

	The present work introduces \emph {ConicCurv}, a new \emph{derivative-free} algorithm to estimate the curvature of a plane curve from a sample of data points that aims to be consistent with the geometry suggested by the data.
	
	Our proposed method assigns curvature values at each point of an ordered set of data in the plane, calculating the average of the curvatures at the point of two conic curves. Both curves interpolate the given point as well as the previous and the next neighbors, without computing the equations of these interpolating conics. The first interpolating curve is chosen as the rational conic in Bernstein-B\'ezier form that, additionally to the three points, interpolates the tangent directions assigned to the previous point and the point in question. Analogously, the second interpolating curve is chosen as the rational conic in Bernstein-B\'ezier form that, additionally to the three points, interpolates the tangent directions assigned to the point in question and its subsequent neighbor.
	
	If the tangent directions are not provided as data, we estimate them using the tangent estimation method in \cite{Albrecht2}. In what follows, we refer to this tangent estimation method as ABFH. By coupling ABFH with \emph {ConicCurv} very desirable properties are inherited, such as high convergence order, reproduction of exact curvature values, if the points are sampled from conics, invariance under Euclidean transformations, and local control of the geometry.
	
	In Section \ref{sec:preprocessing} we show the preprocessing of the data before curvature estimation. Part of that preprocessing consists on splitting the polygonal defined by the data in convex sub-polygons whose vertices are a subset of the original data points in the previously specified order. If there are ``many'' convexity changes (that is, a considerable number of sub-polygons), the data can be smoothed and then tangents are estimated at each point.
	
	Taking into account the preprocessing of Section \ref{sec:preprocessing}, in Section \ref{sec:curvcomp} curvature values are assigned at each point with a computational cost of $O(n)$, {where $n$ is the number of data points}, for data preprocessing and curvature estimation with \emph{ConicCurv}. In Section \ref{sec:apporder} the approximation order of the proposed curvature estimation is shown. In Section~\ref{sec:num_ex} the numerical results obtained with ConicCurv are compared with those obtained using some methods
	of curvature estimation reported in the literature \cite{Be99,Lewiner}.
	
	Finally, in Sections \ref{sec:Lcurv} and \ref{snake} applications of \emph {ConicCurv} to the estimation of the elastic energy of subdivision curves and to the estimation of the location of L-curve corners are reported, respectively.

	\section{Data preprocessing}
	\label{sec:preprocessing}
	
	Given an ordered set of points in the plane, in addition to the coordinates of the points, the \emph{ConicCurv} method also requires tangent directions assigned to these points. If the tangent directions are not available \emph{a priori}, the method of assigning tangent directions that we used, as shown in Section \ref{subsec:tangent}, assumes the convexity of the polygon formed by the data. Therefore, the first step to be carried out in the preprocessing is the splitting of the initial polygon into convex sub-polygons; see Section~\ref{subsec:anal-de-conv}. We assume that there are not three consecutive collinear points in the given set. {In that case, we set the curvature in the middle point equals to zero and no estimation is performed. This assignment is consistent with the fact that the straight lines are the only irreducible algebraic curves which have an arc that interpolates three collinear points.}
	
	As notation, by $\mathcal{P} = \{P_{i}, i = 1, \ldots, n\}$ we refer to an ordered set of points and, at the same time, to the polygon having those $n$ points as vertices in given order. Consequently, $\mathcal{P}(s:r) = \{P_i, i=s,\ldots,r\}$, with $1\leq s \leq r \leq n$, $\mathcal{P}(i) = P_i$, and $\lvert \mathcal{P} \rvert = n$ denotes the cardinality of the set $\mathcal{P}$.

	\subsection{\emph{Pascal's theorem} and tangent estimation}
	\label{subsec:tangent}
	
	There are several algorithms for estimating tangent vectors given a set of points in the plane $\mathbb{R}^2$. If the polygon that joins pairs of consecutive points in the given order is convex, it is shown that ABFH has an order of approximation $4$, higher than other methods in the specialized literature. This tangent estimator algorithm is based on \emph{Pascal's theorem}.
	
	\begin{theorem}[Pascal's theorem]
		\label{theo:Pascal}
		The three pairs of opposite sides of an hexagon inscribed in a conic section intersect at three collinear points.
	\end{theorem}
	
	\begin{figure}[h]
		\centering
		\scalebox{.5}{
			\begin{tikzpicture}	
				\clip (-3.2,-4.5) rectangle (8.5,2);
				\draw[thick,domain=-3:3,samples=100] plot (\x,{-.52*(\x)^2});
				
				\foreach \c/\l in {-2.7/1,2.7/6,-1.9/2,1.9/5,-.2/3,1./4}
				\coordinate (\l) at (\c,-.52*\c*\c);
				
				\node[below right=1mm] at (1) {\large $P_{i-2}$};
				\node[right=1mm] at (2) {\large $P_{i-1}$};
				\node[below=1mm] at (3) {\large $P_i$};
				\node[above right=0mm] at (4) {\large $P_{i+1}$};
				\node[right=1mm] at (5) {\large $P_{i+2}$};
				\node[below left=1mm] at (6) {\large $P_{i+3}$};
				
				\draw[thin,gray,variable=\t,name path=line 12] plot[domain=-3.5:10] (\t,{\myline(-2.7,-.52*2.7*2.7)(-1.9,-.52*1.9*1.9)});
				\draw[thin,gray,variable=\t,name path=line 45] plot[domain=-3.5:10] (\t,{\myline({1.1},{-.52*1.1*1.1})({1.9},{-.52*1.9*1.9})});
				\shade[ball color=red,name intersections={of=line 12 and line 45, by={a}}] (a) node[left=1mm,black] {\large $a_i$} circle (2.5pt);
				
				\draw[thin,gray,variable=\t,name path=line 23] plot[domain=-3.5:10] (\t,{\myline({-.2},{-.52*.2*.2})({-1.9},{-.52*1.9*1.9})});
				\draw[thin,gray,variable=\t,name path=line 56] plot[domain=-3.5:10] (\t,{\myline({2.7},{-.52*2.7*2.7})({1.9},{-.52*1.9*1.9})});
				\shade[ball color=red,name intersections={of=line 23 and line 56, by={b}}] (b) node[right=1mm,black] {\large $b_i$} circle (2.5pt);
				
				\draw[thin,gray,variable=\t,name path=line 34] plot[domain=-3.5:10] (\t,{\myline({-.2},{-.52*.2*.2})({1.1},{-.52*1.1*1.1})});
				\draw[thin,gray,variable=\t,name path=line 16] plot[domain=-3.5:10] (\t,{\myline({-2.7},{-.52*2.7*2.7})({2.7},{-.52*2.7*2.7})}); 
				\shade[ball color=red,name intersections={of=line 34 and line 16, by={c}}] (c) node[above=1mm,black] {\large $c_i$} circle (2.5pt);
				
				\draw[red] (a) -- (c);
				\shade[ball color=blue] (a) circle (2.5pt);
				\shade[ball color=blue] (b) circle (2.5pt);
				\shade[ball color=blue] (c) circle (2.5pt);
				
				\foreach \c in {-2.7,2.7,-1.9,1.9,-.2,1.1}
				\shade[ball color=blue] (\c,-.52*\c*\c) circle (2.5pt);
			\end{tikzpicture}
		}
		\caption{Pascal's theorem}
		\label{fig:pascal}
	\end{figure}
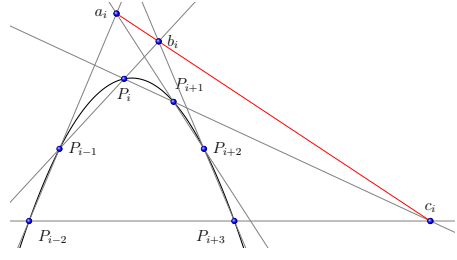
	
	The \emph{principle of duality} of Projective Geometry allows us to express the union of points, that is, the line that joins them, and the intersection of lines by the vector product (also known as exterior or cross product $\wedge$ \cite{Farin}). It holds that
	\begin{math}
		L_{P_{i},P_{j}} = P_{i} \wedge P_{j}
	\end{math}
	is the line that passes through the points $P_i$ and $P_j$, and
	\begin{math}
		L_{P_{i},P_{j}} \wedge L_{P_{k},P_{l}}
	\end{math}
	is the intersection point of lines $L_ {P_{i},P_{j}}$ and $L_{P_{k},P_{l}}$. Then, Theorem \ref{theo:Pascal} can also be stated in the following way.
	
	\begin{theorem}
		\label{theo:Pascal2}
		Let $P_{i-2}, P_{i-1},P_{i}, P_{i+1}, P_{i+2}, P_{i+3}$ be six distinct points \emph{in general position}, i.e., such that there is no line interpolating three of them, then it holds that the intersection points $a_i$, $b_i$, and $c_i$, defined by
		\begin{displaymath}
			a_i = L_{P_{i-2}, P_{i-1}} \wedge L_{P_{i}, P_{i+1}}, \quad
			b_i = L_{P_{i-1},P_{i}} \wedge L_{P_{i+1}, P_{i+2}}, \quad
			\mbox{and} \quad
			c_i= L_{P_i, P_{i+1}} \wedge L_{P_{i-2}, P_{i+2}},
		\end{displaymath}
		are \emph{collinear} points; see Fig. \ref{fig:pascal}.
	\end{theorem}

	\begin{remark}
		Let us note that the point $c_i$ can be computed once $a_i$ and $b_i$ have been obtained, as $c_i= L_{P_{i-2}, P_{i+2}} \wedge L_{a_i,b_i}$. This approach is exploited in the following.
	\end{remark}

	\begin{figure}[h]
		\centering
		\scalebox{.5}{
			\begin{tikzpicture}
				\clip (-3.2,-4.5) rectangle (8.5,2);
				\draw[thick,domain=-3:3,samples=100] plot (\x,{-.52*(\x)^2});
				
				\foreach \c/\l in {-2.7/1,2.7/5,-1.9/2,1.9/4,.5/3}
				\coordinate (\l) at (\c,-.52*\c*\c);
				
				\node[below right=1mm] at (1) {\large $P_{i-2}$};
				\node[right=1mm] at (2) {\large $P_{i-1}$};
				\node[below=1mm] at (3) {\large $P_i$};
				\node[right=1mm] at (4) {\large $P_{i+1}$};
				\node[below left=1mm] at (5) {\large $P_{i+2}$};
				
				\draw[thin,gray,variable=\t,name path=line 12] plot[domain=-3.5:10] (\t,{\myline(-2.7,-.52*2.7*2.7)(-1.9,-.52*1.9*1.9)});\draw[thin,gray,variable=\t,name path=line 34] plot[domain=-3.5:10] (\t,{\myline(.5,-.52*.5*.5)(1.9,-.52*1.9*1.9)});
				\shade[ball color=red,name intersections={of=line 12 and line 34, by={a}}]
				(a) node[left=1mm,black] {\large $a_i$} circle (2.5pt);
				
				\draw[thin,gray,variable=\t,name path=line 23] plot[domain=-3.5:10] (\t,{\myline(.5,-.52*.5*.5)(-1.9,-.52*1.9*1.9)});
				\draw[thin,gray,variable=\t,name path=line 45] plot[domain=-3.5:10] (\t,{\myline(2.7,-.52*2.7*2.7)(1.9,-.52*1.9*1.9)});
				\shade[ball color=red,name intersections={of=line 23 and line 45, by={b}}]
				(b) node[right=1mm,black] {\large $b_i$} circle (2.5pt);
				
				\draw[thin,gray,variable=\t,name path=line 15] plot[domain=-3.5:10] (\t,{\myline(-2.7,-.52*2.7*2.7)(2.7,-.52*2.7*2.7)}); 
				\draw[thin,gray,name path=line ab] (a) ++($3*(a)-3*(b)$) -- ++($-15*(a)+15*(b)$);
				\shade[ball color=red,name intersections={of=line 15 and line ab, by={c}}]
				(c) node[above=1mm,black] {\large $c_i$} circle (2.5pt);
				
				\draw[ultra thick] (3) ++($3*(3)-3*(c)$) -- ++($-15*(3)+15*(c)$);
				
				\foreach \c in {-2.7,2.7,-1.9,1.9,.5}
				\shade[ball color=blue] (\c,-.52*\c*\c) circle (2.5pt);
				
				\shade[ball color=red] (a) circle (2.5pt);
				\shade[ball color=red] (b) circle (2.5pt);
				\shade[ball color=red] (c) circle (2.5pt);
			\end{tikzpicture}
		}
		\caption{Tangent estimation in $P_i$. This tangent is the limit of the secant $L_{P_i,P_{i+1}}$ in Fig. \ref{fig:pascal}.}
		\label{tangente}
	\end{figure}
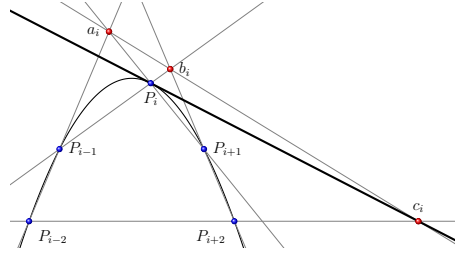
	
	In \cite{Albrecht2} the tangent line of the conic interpolating $P_{j}$, $j = i-2, \ldots, i+2$, at the middle point, $P_i$, is estimated as a limit case of Pascal's theorem \ref{theo:Pascal2}. When two consecutive points collapse to one point, then the tangent line at this point is the limit of the secant between the two collapsing points; see Fig. \ref{tangente}. Hence, the tangent line $r_i$ at $P_i$ of the conic interpolating $P_{i-2}, P_{i-1},P_{i}, P_{i+1}, P_{i+2}$ is $r_i=L_{c_i,P_i}$.
	
	Given the points $\{P_{j}, j = i-2, \ldots, i+2\}$, the tangent line at $P_i$ of the conic interpolating those 5 points is computed using Algorithm \ref{alg:estima5pts}.

	\begin{Algorithm}[H]
		\begin{algorithmic}[1]
			\Require $\{P_{i-2}, \ldots, P_{i+2}\}$
			\State $a_i= L_{P_{i-2},P_{i-1}} \wedge L_{P_i,P_{i+1}}$
			\State $b_i = L_{P_{i-1},P_i} \wedge L_{P_{i+1},P_{i+2}}$
			\State $c_i = L_{P_{i-2},P_{i+2}} \wedge L_{a_i,b_i}$
			\State $\br_{i} = L_{c_i,P_i}$
			\Ensure $\br_{i}$
		\end{algorithmic}
		\caption{Tangent line estimation in $P_i$ with the method \emph{ABFH}, proposed in \cite{Albrecht2}.}
		\label{alg:estima5pts}
	\end{Algorithm}
	
	The tangent lines at the other points can be estimated by changing the ordering.
	
	\begin{remark}
		\begin{itemize}
			\item[$\triangleright$] The computational cost of Algorithm \ref{alg:estima5pts} is {$66$ flops}, since it only requires the computation of the equations of $7$ lines joining $2$ points and to solve $3$ systems of $2$ linear equations (the intersection of $3$ pairs of such lines).
			\item[$\triangleright$] Compared to the method called \emph{Conic}, which is the method that consists in computing the implicit equation of the conic interpolating the points $\{P_{j},\,j = i-2, \ldots, i+2\}$ and calculating from this equation the tangent at one point \cite{Albrecht2},
			Algorithm \ref{alg:estima5pts} needs fewer computations. The computational cost of \emph{Conic} is {$164$ flops}, since it requires to solve a system of $5$ linear equations and to compute the equation of the tangent line to the conic at $P_i$.
		\end{itemize}	
	\end{remark}
	
	If we assume that the data $P_1, \ldots, P_n$ represent a closed convex polygon, then we estimate the tangent in each point $P_i$ by taking as input for the algorithm the points $P_{i-2}$, $P_ {i -1}$, $P_{i}$, $P_{i + 1}$, $P_{i + 2}$, where $P_{-1} = P_{n-1}$, $P_ {0} = P_{n}$, $P_{n + 1} = P_{1}$, and $P_{n + 2} = P_{2}$. On the other hand, if they represent an open convex polygon, the same procedure is followed except for the extreme points $P_{1}, P_{2}, P_{n-1}$, and $P_{n}$. To find $ r_1$ and $r_2$ , the points $P_1, P_2, P_3, P_4, P_5$ are reordered so that $P_1$ and $P_2$ are, according to the respective case, the third point in the input of the algorithm. Similarly, to find $ r_{n-1}$ and $ r_{n}$, $P_{n-4}, P_{n-3}, P_{n-2}, P_{ n-1}, P_n$ are reordered conveniently for each case.
	In the implementation, it is convenient to work with projective coordinates, so it makes sense to consider the point of intersection even for parallel lines (that intersect at an \emph{improper point}{, i.e., a point at \emph{infinity}}) and take advantage of the duality in the computation of $r_i \wedge r_j$ ($r_i, r_j$ straight lines) and $P_i \wedge P_j$ ($P_i, P_j$ points), having just one method for both operations.
	
	Since Algorithm \ref{alg:estima5pts} is based on finding intersections of lines, it is invariant under affine transformations. By assigning at each point the tangent of the conic that interpolates it together with its $4$ closest neighboring points, it can be concluded that the exact tangent directions are recovered if the data are sampled on a conic.
	
	\subsection{Convexity analysis of the data}
	\label{subsec:anal-de-conv}
	
	The method of estimating tangents using Algorithm~\ref{alg:estima5pts} may lead to erroneous results if the polygon $\mathcal{P} = \{P_{i}, i = 1, \ldots, n \}$ is non-convex \cite{Albrecht2}; see Fig. \ref{fig:pol_noconvex_gl}. The solution to this problem is to split $\mathcal{P}$ into convex sub-polygons. In this way, tangents can be estimated independently in each sub-polygon. This splitting can be done by means of two approaches.
	The \emph{first approach}, proposed in \cite{Albrecht:convexity}, consists in inserting an \emph{inflexion point}, $P$, at the midpoint of the inflection edge $P_{i} P_{i + 1}$, as shown in Fig. \ref{fig:pol_noconvex_gl}. Then, the left tangent direction, $\bt_{P, l}$, and the right tangent direction, $\bt_{P, r}$, at $P$ are estimated associated to the corresponding sub-polygons, that is to say, taking the points $\{P_{i-3}, \ldots, P_{i}, P\}$ and  $\{P, P_{i + 1}, \ldots, P_{i + 4}\}$, as entries for Algorithm \ref{alg:estima5pts}, respectively. Finally, the bisector of the acute angle between them is defined as tangent vector at $P$; see Fig. \ref{fig:pol_noconvex_gl}. This can be obtained as a convex linear combination of these two tangent vectors.

	\begin{figure}[ht]
		\centering
		\begin{tikzpicture}[scale=.8]	
			\foreach \x/\y [count=\i] in {0/1, 1/2.3, 3/2.5, 4.7/1.7, 5.3/.2, 7/0, 8.7/.7, 9.2/2.3}{
				\coordinate (\i) at (\x,\y);
			}
			
			\node[left] at (1) {$P_{i-3}$};
			\node[above left] at (2) {$P_{i-2}$};
			\node[above] at (3) {$P_{i-1}$};
			\node[below left] at (4) {$P_i$};
			\node[below] at (5) {$P_{i+1}$};
			\node[below] at (6) {$P_{i+2}$};
			\node[below right] at (7) {$P_{i+3}$};
			\node[right] at (8) {$P_{i+4}$};
			
			\foreach \tx/\ty [count=\i] in {-0.2016/0.0085, -0.2099/0.0520}{
				\coordinate (t\i) at ($1/veclen(\tx,\ty)*(\ty,-\tx)$);
			}
			\draw[red,thick] ($(4)!.5!(5) - (t1)$) -- +($2*(t1)$) node[black, above left=-3pt] {$\bt_{P,l}$};
			\draw[red,thick] ($(4)!.5!(5) - (t2)$) -- +($2*(t2)$) node[black,above right=-3pt] {$\bt_{P,r}$};
			\draw[densely dotted,thick] ($(4)!.5!(5) - .6*(t1)-.6*(t2)$) -- +($1.2*(t1)+1.2*(t2)$);
			
			\draw (1) \foreach \i in {2,...,8}{ -- (\i)};
			\draw[ultra thick,green!50!blue] (4) -- (5);
			
			\foreach \i in {1,...,8}{
				\shade[ball color = blue] (\i) circle (2pt);
			}
			\shade[ball color = red] ($(4)!.5!(5)$) circle (2pt) node[black,right] {$P$};
		\end{tikzpicture}
		\caption{Nonconvex polygon divided into two sub-polygon inserting the inflection point $P$. The assigned tangent is drawn with discontinuous lines.}
		\label{fig:pol_noconvex_gl}
	\end{figure}
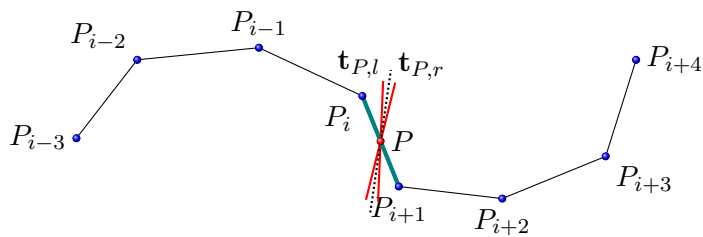
	
	To avoid inserting more data than the user proposes, a \emph{second approach} consists in dividing the initial polygon into independent convex sub-polygons, considering the inflection edge $\overline{P_i P_{i+1}}$ as belonging to the two sub-polygons that contain it as a final or initial edge, respectively; see Fig. \ref{fig:pol_noconvex_gl}. At each point, the tangent is estimated considering only the sub-polygon to which it belongs.
	
	Algorithm \ref{alg:subpol_convex} illustrates the proposal, with $R_{\frac{\pi}{2}}$ denoting a rotation of $\frac\pi2$ radians, $\langle\cdot,\cdot\rangle$ the usual scalar product, and $sign(x) = \frac{\lvert x \rvert}{x}$ the sign function (discarding the case $x=0$, as there are not three consecutive collinear points in $\mathcal{P}$).
	
	\begin{Algorithm}[H]
		\begin{algorithmic}[1]
			\Require $\mathcal{P}=\{P_{i}, i = 1, \ldots, n\}$
			\State $j = 1$;
			\While{($\lvert\mathcal{P}\rvert\geq 3$)} \Comment{The smallest amount of vertices in a sub-polygon is 3.}
			\State $\mathcal{PC}_j(1:3) = \mathcal{P}(1:3) $;
			\State $s = \,sign(\,\langle P_1 - P_2,R_{\frac{\pi}{2}}(P_3 - P_2) \rangle\,)$;
			\State $i = 3$;
			\State $n = \lvert \mathcal{P} \rvert$;
			\While{($\,i < n \,$) \AND ($\,sign(\,\langle P_{i-1} - P_i,R_{\frac{\pi}{2}}(P_{i+1} - P_i) \rangle\,) == s$)}
			\State $\mathcal{PC}_j(i+1) = \mathcal{P}(i+1)$;
			\State $i = i+1$;
			\EndWhile
			\State $\mathcal{P} = \mathcal{P}(i-1:n)$;
			\State $j = j+1$;
			\EndWhile
			\Ensure $\{\mathcal{PC}_1,\mathcal{PC}_2, \ldots,\mathcal{PC}_j\}$
			\end
			{algorithmic}
			\caption{Splitting the initial polygon $\mathcal{P}$ into convex sub-polygons $\{\mathcal{PC}_j\}$ (second approach).}
			\label{alg:subpol_convex}
		\end{Algorithm}

		There are cases in which there is an insufficient amount of input data for the Algorithm \ref{alg:estima5pts}, that is, there may be sub-polygons with less than $5$ vertices. To solve this, in \cite{Estrada-Diaz} the strategy followed by Albrecht et al. has been generalized, collapsing at a certain point a pair of points (selected in a convenient way) and considering as a straight line joining the points of each pair, a tangent assigned a priori by the designer. This method is also based on finding solutions of systems of $2$ linear equations of and does not require the calculation of the implicit equation of the conic.
		
		Only two cases considered in \cite{Diaz-Estrada} should be additionally considered: when a sub-polygon has $3$ or $4$ points. The case of $2$ points is discarded, because the convexity is analyzed in polygons of at least $3$ points. Alternatively, in \cite{Albrecht:convexity} a way of inserting new points in the polygon is proposed, so that each sub-polygon has at least the $5$ points necessary for the estimation of the tangents. This is a rule for automatic point insertion and it is not dependent on the design. It differentiates the cases of $3$ and $4$ points and ensures that the convexity of the sub-polygon with the new points remains unaffected.

		\section{Curvature computation}
		\label{sec:curvcomp}
		
		In this section we show how to efficiently estimate the curvature of a smooth curve $\bc$ at a sample of points $\{P_j\}$ using the set $\{ {r}_j \}$ of tangent lines associated to $\{ P_j \}$ by ABFH.
		
		Based on a general theorem of E. Cartan \cite{Car35}, which states that two curves are related by a group transformation if and only if their signature curves are identical, in \cite{Ca98}  the differential invariant signature curve paradigm was introduced for the invariant recognition of visual objects. In visual applications, a group transformation $G$ (typically either the Euclidean, affine, similarity, or projective group) acts on a space $E$ representing the image space, whose subsets are the objects of interest. A differential invariant $I$ of $G$ is a real-valued function, depending on the curve and its derivatives at a point, which is unaffected by the action of $G$. Consequently, to construct a numerical approximation to a differential invariant $I$, that approximation should be computed using appropriate combinations of the coordinates of the sample points. This idea draws a bridge between the discrete and continuous invariant theory. The first example discussed in \cite{Ca98} are the Euclidean plane curves (where the simplest differential invariant of the Euclidean group is the Euclidean curvature) and the $3$-\emph{points} invariant numerical approximations to the Euclidean curvature.
		
		For any three noncollinear points in the plane $P_1, P_2$, and $P_3$, let us denote by $A[P_1P_2P_3]$ the signed area of $\triangle P_1P_2P_3$, i.e., $A[P_1P_2P_3]= \frac{(P_i-P_j) \wedge (P_i-P_k)}{2} $ (the area is positive if the triangle is traversed in a clockwise direction).
		The quantity $A[P_1P_2P_3]$ is the simplest invariant under the action of the special affine group $SA(2)$ consisting of all area-preserving affine transformations of the plane. According to \cite{We46}, every joint affine invariant $I (P_1, \ldots, P_n)$ depending on the $n$ points $P_i$ is a function of these triangular areas.
		It its well known that five points $P_0, \ldots,P_4$ in general position in the plane determine a unique
		conic section that passes through them. The implicit formula for the $5$-\emph{points} interpolating conic in terms of the invariants $A[P_iP_jP_k]$ is a classical result that may be found in \cite{Ca98,St93}. In a similar way, given three points $P_i,P_j,P_k$ and tangent lines associated to two of them, we can establish the $SA(2)$-invariant form of the unique conic interpolating these points and the associated tangent directions.
		
		\begin{lemma} \label{lem:invariant}
			Let $P_i,P_j,P_k$ be three points in general position in the plane and let ${r}_i$, ${r}_k$ be the tangent lines associated to $P_i,P_k, i\neq k$, respectively. Let $Q=r_i \bigwedge r_k$. Then, the unique conic $\mathcal{C}$ interpolating $P_i,P_j,P_k$ as well as the tangent directions associated to $P_i$, $P_k$, $i\neq k$, satisfies the quadratic $SA(2)$-invariant implicit equation
			\begin{equation}
				A[\mathbf{x}P_iP_k]^2 A[P_jQP_i] A[P_jP_kQ] = A[P_jP_iP_k]^2 A[\mathbf{x}P_kQ] A[\mathbf{x}QP_i],
				\label{ecuacion}
			\end{equation}
			where $\mathbf{x}=(x,y)$ is an arbitrary point on $\mathcal{C}$.\\
			Moreover, if $\triangle P_iP_kQ$ is nondegenerate, the curvature of the unique conic $\mathcal{C}$ at $P_i$, $\kappa(P_i)$,
			has the $SE(2)$-invariant formula
			\begin{equation}
				\kappa(P_i) = 4 \, \frac{A[P_iP_kQ] A[P_jP_kQ] A[P_jQP_i]}{{A[ P_jP_iP_k]}^2 \, {\|Q-P_i\|}^3}.
				\label{ecuacion2}
			\end{equation}
		\end{lemma}

		\begin{proof}
			We write $\mathbf{x}=(x,y)$ and $P_{j}$ in barycentric coordinates with respect to $\triangle P_{i}QP_k$ ,
			\begin{equation*}
				\mathbf{x} = (x,y) = u P_{i} + v Q + w P_{k}
				\quad \text{and} \quad
				P_{j} = u_j P_{i} + v_j Q + w_j P_{k},
			\end{equation*}
			with $u+v+w = u_j + v_j + w_j = 1$. Since $\mathbf{x}=(x,y)$ is an arbitrary point on $\mathcal{C}$, its barycentric coordinates satisfy the implicit equation
			\begin{equation}
				v^2=4\, \Omega^2\, u \,w,
				\label{df2}
			\end{equation}
			with
			\begin{equation}
				\Omega^2=\frac{v_j^2}{4 \, u_j\,w_j}.
				\label{Omega2}
			\end{equation}
			
			Recall the geometric interpretation of the barycentric coordinates with respect to $\triangle P_{i}QP_k $
			\begin{equation}\label{u}
				\begin{split}
					&u=\frac{A[\mathbf{x}P_k Q ]}{A[P_i P_k Q]}, \quad v=\frac{A[\mathbf{x}P_i P_k]}{A[P_i P_k Q]}, \quad w=\frac{A[\mathbf{x}Q P_i ]}{A[P_i P_k Q]}, \\
					&u_j=\frac{A[P_j P_k Q ]}{A[P_i P_k Q]}, \quad
					v_j=\frac{A[P_j P_i P_k]}{A[P_i P_k Q]}, \quad\text{and}\quad
					w_j=\frac{A[P_j Q P_i ]}{A[P_i P_k Q]}.
				\end{split}
			\end{equation}

			Substituting \eqref{u} in \eqref{df2}, we obtain the $SA(2)$-invariant implicit equation \eqref{ecuacion} of the unique conic $\mathcal{C}$ determined by the 4-\emph{points} $P_i,P_j,P_k$ and $Q$.
			
			Given a rational conic $\mathcal{C}$ in Bernstein-B\'ezier form with control polygon $P_i, Q, P_k$ and parameter
			$\Omega$, 
			then the well known formula of the curvature of $\mathcal{C}$ at $P_i$, $\kappa(P_i)$, is \cite{Farin}
			\begin{equation}
				\kappa(P_i)= \frac{{A[P_iP_kQ]}}{{\Omega}^2\,{\|Q-P_i\|}^3}.
				\label{ecuacion1}
			\end{equation}
			
			Substituting the expressions for $u_j,v_j$ and $w_j$ from \eqref{u} in \eqref{Omega2}, according to \eqref{ecuacion1}
			we obtain the $SE(2)$-invariant formula \eqref{ecuacion2}.
		\end{proof}
		
		Given a sample of points $\{P_j\}$ on a smooth curve $\bc$ and the set $\{ {r}_j \}$ of tangent lines associated to $\{ P_j \}$ by means of ABFH, we denote by $\bc_{i-1,r}$ the unique conic interpolating the points $P_j,\,j=i-1,i,i+1$ and the tangent vectors $\tau_j$ at $P_j$ corresponding to $r_j, \, j=i,i+1$. Analogously, we denote by $\bc_{i,l}$ the unique conic interpolating the points $P_j,\,j=i-1,i,i+1$ and the tangent vectors at $P_j$, $\tau_j, \, j=i-1,i$. To each point $P_i$ are associated two auxiliary points $Q_j= r_{j+1}\bigwedge r_j,\,j=i-1,i$.
		
		According to the previous Lemma \ref{lem:invariant}, in order to compute the curvatures at $P_i$ of the interpolating conics $\bc_{i,l}$ and $\bc_{i-1,r}$ it is only necessary to compute the auxiliary points $Q_{i-1}$, $Q_i$, the distances $\|Q_{i-1}-P_i\|$, $\|Q_{i}-P_i\|$ and the areas $A[ P_{i-1}P_iP_{i+1}]$, $A[P_iP_{i+1}Q_{i-1}]$, $A[P_{i-1}P_{i+1}Q_{i-1}]$, $A[P_{i-1}Q_{i-1}P_i]$, $A[P_iP_{i-1}Q_{i}]$, $A[P_{i+1}P_{i-1}Q_{i}]$, $A[P_{i+1}Q_{i}P_i]$.

		\begin{definition} Let $\kappa_{i,l}(P_i)$ and $\kappa_{i-1,r}(P_i)$ be the curvature values of the interpolating conics $\bc_{i,l}$ and $\bc_{i-1,r}$ at $P_i$. Our estimator $\widetilde{\kappa}(P_i)$ for the curvature of the curve $\bc$ at $P_i$ is defined as
			\begin{equation}
				\widetilde{\kappa}(P_i) = \frac{\kappa_{i,l}(P_i) + \kappa_{i-1,r}(P_i)}{2}.
				\label{konda}
			\end{equation}
		\end{definition}
		
		\begin{Algorithm}[H]
			\begin{algorithmic}[1] \onehalfspacing
				\Require $j,\,\, \mathcal{P}(j-3, \ldots,j+3)$
				\For {$i=j-1,j,j+1$}
				\State $r_i=\texttt{Algorithm}\,\ref{alg:estima5pts}\,(P_{i-2}, \ldots,P_{i+2})$;
				\EndFor
				\State $Q_{j-1}= r_{j-1} \bigwedge r_j$;
				\State $ Q_{j}= r_{j} \bigwedge r_{j+1}$;
				\State Compute $A[ P_{j-1}P_jP_{j+1}]$;
				\State Compute$A[P_jP_{j+1}Q_{j-1}]$, $A[P_{j-1}P_{j+1}Q_{j-1}]$, $A[P_{j-1}Q_{j-1}P_j]$, and $\|Q_{j-1}-P_j\|$; \vskip 3pt
				
				\State $\kappa_{l}(P_j)= 4 \, \frac{A[P_jP_{j+1}Q_{j-1}]A[P_{j-1}P_{j+1}Q_{j-1}]A[P_{j-1}Q_{j-1}P_j] }{{A[ P_{j-1}P_jP_{j+1}] }^2 \, {\|Q_{j-1}-P_j\| }^3}$ ; \vskip 3pt
				
				\State Compute $A[P_jP_{j+1}Q_{j}]$, $A[P_{j-1}P_{j+1}Q_{j}]$, $A[P_{j-1}P_jQ_{j}]$, and $\|Q_{j}-P_j\|$; \vskip 3pt
				
				\State $\kappa_{r}(P_j)= 4 \, \frac{A[P_jP_{j+1}Q_{j}]A[P_{j-1}P_{j+1}Q_{j}]A[P_{j-1}P_jQ_{j}]}{{A[ P_{j-1}P_jP_{j+1}] }^2 \, {\|Q_{j}-P_j\| }^3}$; \vskip 3pt
				\State $\widetilde{\kappa}(P_j)=\frac{\kappa_{l}(P_j)+\kappa_{r}(P_j)}{2}$;
				\Ensure $\widetilde{\kappa}(P_j)$
				\end
				{algorithmic}
				\caption{Curvature estimation $\widetilde{\kappa}(P_j)$ with \emph{ConicCurv} method.}
				\label{alg:AlgCurvature}
			\end{Algorithm}

			The pseudocode in Algorithm~\ref{alg:AlgCurvature} is a sketch of an algorithm to estimate curvatures with \emph{ConicCurv}. For any sample of n points, the computational cost of data preprocessing and ConicCurv is clearly $O(n)$. 
			{In particular, for one run, (i.e., estimating the curvature in the center point of 7 ordered points), the computational cost of Algorithm~\ref{alg:AlgCurvature} is $304$ flops.}
			An efficient implementation should reuse the intermediate computations of the estimations of tangent lines and curvatures at the previous points obtained with ABFH and \emph{ConicCurv}.
			
			\begin{figure}[h!]
				\centering
				\subcaptionbox{$\bc_{i-1,r}$}[.45\textwidth][c]{
					\begin{tikzpicture}[scale=.7]	
						\foreach \x/\y [count=\i] in {0/0,3/3,7/0}{
							\coordinate (\i) at (\x,\y);
						}
						\foreach \x/\y [count=\i] in {1/6,4/1,-1/4}{
							\coordinate (t\i) at ($1/veclen(\x,\y)*(\x,\y)$);
						}
						
						\node[below right] at (1) {$P_{i-1}$};
						\node[below=1mm] at (2) {$P_i$};
						\node[below left] at (3) {$P_{i+1}$};
						
						\coordinate (q2) at (6.0588,3.7647); 
						
						\draw[dashed] (1) -- (2) -- (3);
						\draw (2) -- (q2) -- (3);
						
						\draw[domain=-1.2:1,variable=\t,samples=100,black,densely dotted] let \p1 = (2),\p2 = (q2),\p3 = (3) in \parambernbez{\p1}{\p2}{\p3}{.779};

						\foreach \i in {2,3}{ 
							\draw[red,very thick] ($(\i) - .5*(t\i)$) -- +(t\i);
						}
						\foreach \c in {1,2,3}
						\shade[ball color = blue] (\c) circle (2pt);
						\shade[ball color = red] (q2) circle (2pt) node[above,black] {$Q_{i}$};	
					\end{tikzpicture}
				}
				\hfill
				\subcaptionbox{$\bc_{i,l}$}[.45\textwidth][c]{
					\begin{tikzpicture}[scale=.7]	
						\foreach \x/\y [count=\i] in {0/0,3/3,7/0}{
							\coordinate (\i) at (\x,\y);
						}
						\foreach \x/\y [count=\i] in {1/6,4/1,-1/4}{
							\coordinate (t\i) at ($1/veclen(\x,\y)*(\x,\y)$);
						}
						
						\node[below right] at (1) {$P_{i-1}$};
						\node[below=1mm] at (2) {$P_i$};
						\node[below left] at (3) {$P_{i+1}$};
						
						\coordinate (q1) at (0.3913, 2.3478); 
						
						\draw[dashed] (1) -- (2) -- (3);
						\draw (1) -- (q1) -- (2);
						
						\draw[domain=0:5.7,variable=\t,samples=100,black,densely dotted] let \p1 = (1),\p2 = (q1),\p3 = (2) in \parambernbez{\p1}{\p2}{\p3}{.8016};	
						
						\foreach \i in {1,2}{ 
							\draw[red,very thick] ($(\i) - .5*(t\i)$) -- +(t\i);
						}
						\foreach \c in {1,2,3}
						\shade[ball color = blue] (\c) circle (2pt);
						\shade[ball color = red] (q1) circle (2pt) node[above,black] {$Q_{i-1}$};
					\end{tikzpicture}
				}
				\caption{Interpolating conics.}
				\label{fig:interConics}
			\end{figure}
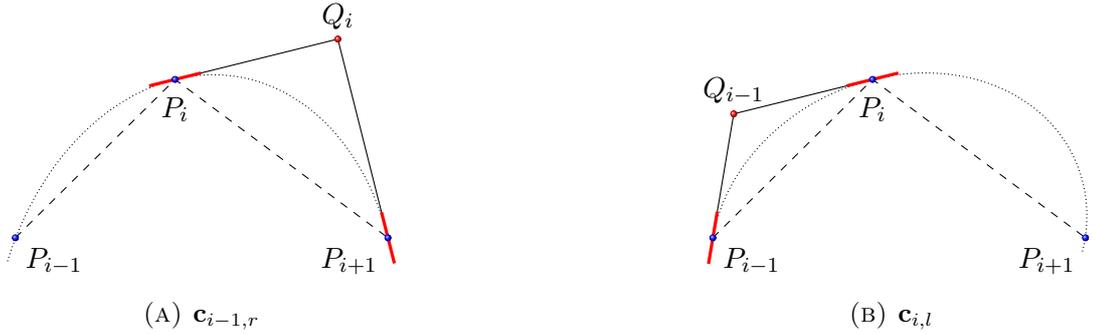
			
			\section{Curvature approximation order}
			\label{sec:apporder}
			
			Now, we show that, given a discrete sample $\{P_i\}$ of points on a sufficiently differentiable parametric curve $\bc(s)$, the average of the curvatures at $P_i$ of the interpolating conics $\bc_{i,l}$ and $\bc_{i-1,r}$, i.e., $\widetilde{\kappa}(P_i)$, provides a good numerical approximation to the curvature of $\bc$ at the point $P_i$.

			Let $P_{j}=\bc(s_j)$ and $\bt_j$ be the unitary tangent vector of $\bc(s)$ at $P_j$. 
			Our extension of ABFH exposed in Section \ref{subsec:anal-de-conv} provides \sefue{unitary} vectors $\btau_j$, that are good approximations to $\bt_j$, in the following sense:
			the \sefue{unitary} tangent vector $\tau_j$ corresponding to the tangent line assigned by ABFH to $ P_j$ satisfies
			\begin{equation}
				\btau_j = \bt_j + O\left(h^4\right) \, \bn_j, \,\,{ \text{as}\; h \rightarrow 0,} 
				\label{O4}
			\end{equation}
			where $h=\max_{\lvert k-j \rvert \leq 2}\,\{ \lvert s_k - s_j \rvert\}$ and $\bn_j$ is the unitary normal vector of $\bc(s)$ at $P_j$ \cite{Albrecht2}.
			
			\begin{theorem}
				Let $\bc(s)$ be a parametric curve, $C^3$-differentiable in a neighborhood of $P_i=\bc(s_i)$ and denote by $\kappa(P_i)$ the curvature of $\bc$ at $P_i$. Then, for $h \rightarrow 0$, the curvature estimator $\widetilde{\kappa}(P_i)$ defined by \eqref{konda} satisfies
				\begin{equation*}
					\kappa(P_i)- \widetilde{\kappa}(P_i)=O(h^3).
				\end{equation*}
				Moreover, if the points $P_{i-1}$, $P_i$, and $P_{i+1}$, are \textit{uniformly} arc-length distributed, it holds
				\begin{equation*}
					\kappa(P_i)- \widetilde{\kappa}(P_i)=O\left(h^4\right).
				\end{equation*}
			\end{theorem}
			
			\begin{proof}
				
				{Let be $P_{j}=\bc(s_j), \, j=1,2,\ldots$ points on the curve $\bc(s)$, $\bt_j$ the unitary tangent vectors of $\bc(s)$ at $P_j$, and $\btau_j$ the vectors associated to the points $P_j$ by our extension of ABFH exposed in Section \ref{subsec:anal-de-conv}. If $h\max_{\lvert k-j \rvert \leq 2} \, \{ \lvert s_k - s_j \rvert\}$}
				is very small, for any fixed index $i$ we may assume that for $j\in\{i,i+1\}$ the arc-length from $P_{j-1}$ to $P_j$ on the three curves $\bc, \bc_{i,l}$ and $\bc_{i-1,r}$ are equal; see Fig. \ref{fig:interConics}. Hence, we may arc-length reparametrize $\bc$, $\bc_{i,l}$ and $\bc_{i-1,r}$, in such a way that it holds
				\begin{align}
					\bc_{i-1,r}(-s_{i-1}) &= \bc(-s_{i-1})=\bc_{i,l}(-s_{i-1}) = P_{i-1}, \label{eq:conicsparL} \\
					\bc_{i-1,r}(0) &= \bc(0) = \bc_{i,l}(0) = P_{i}, \label{eq:conicsZero} \\
					\bc_{i-1,r}(s_{i+1}) &= \bc(s_{i+1})=\bc_{i,l}(s_{i+1}) = P_{i+1}, \label{eq:conicsparR}
				\end{align}
				where $s_i = 0$.
				
				Recall that for any arc-length parameterized function $\br(s)$, if it is smooth at $s=0$ and, denoting by $\,\, '$ the differentiation with respect to the arc-length $s$, then $\bt(0)=\br^{'}(0)$ and $\bn(0)={\bt(0)}^{\bot}$ is the Frenet basis at $s=0$, where $\mathbf{v}^{\bot}$ denotes the vector orthogonal to vector $\mathbf{v}$. According to the Frenet formulas it holds
				\begin{equation*}
					\br^{''}(0)=\bt^{'}(0)=\kappa(0)\bn(0) \quad \text{and} \quad \bn^{'}(0)=-\kappa(0)\bt(0).
				\end{equation*}
				
				Both interpolating conic curves $\bc_{i,l}$ and $\bc_{i-1,r}$ have the same \sefue{unitary} tangent vector at $P_i$, $\btau_i $. Therefore, substituting $\bt_i=\bc^{'}(0) $ in \eqref{O4} we have
				\begin{equation}
					\bc_{i,l}^{'}(0) = \bc_{i-1,r}^{'}(0) = {\btau_i} = \bc^{'}(0) + O\left(h^4\right)\bn_i, \,\,{ \text{as}\; h \rightarrow 0}. \label{eq:conicsIderZero}
				\end{equation}						
				Moreover, let $\bn_{i,l}$ and $\bn_{i-1,r}$ be the unit normal vectors of $\bc_{i,l}$ (respectively of $\bc_{i-1,r}$) at $P_i$. Since $\bn_{i,l}=\bn_{i-1,r}=\btau_i^{\bot},$ from \eqref{O4} it holds
				\begin{equation*}
					\bn_{i,l}=\bn_{i-1,r}=\bn_{i}+O\left(h^4\right)\bt_i,\,\,{ \text{as}\; h \rightarrow 0}.
				\end{equation*}
				Consequently, the principal curvature vectors of the three curves $\bc$, $\bc_{i-1,r}$, and $\bc_{i,l}$, at $P_i$ are $\bc^{''}(0)$, $\bc_{i-1,r}^{''}(0)$, and $\bc_{i,l}^{''}(0)$, respectively, with
				\begin{equation}
					\begin{split}
						\bc^{''}(0) &= \kappa(P_i)\bn_i, \\
						\bc_{i-1,r}^{''}(0) &= \kappa_{i-1,r}(P_i){\btau_i}^{\bot} = \kappa_{i-1,r}(P_i)\bn_i + O\left(h^4\right)\bt_i, \\
						\bc_{i,l}^{''}(0) &= \kappa_{i,l}(P_i){\btau_i}^{\bot} = \kappa_{i,l}(P_i)\bn_i + O\left(h^4\right)\bt_i,
					\end{split} \label{eq:conicIIderparZero}
				\end{equation}
				{ as $h$ tends to 0}, where $\kappa(P_i), \kappa_{i-1,r}(P_i)$, and $\kappa_{i,l}(P_i)$, denote the curvature values at $P_i$ of the curves $\bc$, $\bc_{i-1,r}$, and $\bc_{i,l}$, respectively.
				
				Expanding $\bc(s)$ by its Taylor series around $s=0$, {as $s_{i-1}$ and $s_{i+1}$ tend to 0}, we get
				\begin{equation}	
					\bc(-s_{i-1}) = \bc(0)-s_{i-1}\,\bc^{'}(0) +\frac{{s_{i-1}}^2}{2}\bc^{''}(0) -\frac{{s_{i-1}}^3}{6}\bc^{'''}(0) +O\left({s_{i-1}}^4\right) \label{eq:conicTaylorparL}
				\end{equation}
				and
				\begin{equation}
					\bc(s_{i+1}) = \bc(0)+s_{i+1}\,\bc^{'}(0) +\frac{{s_{i+1}}^2}{2}\bc^{''}(0) +\frac{{s_{i+1}}^3}{6}\bc^{'''}(0) + O\left({s_{i+1}}^4\right).
					\label{eq:conicTaylorparR}
				\end{equation}
				In a similar way, expanding the interpolating conics by their Taylor series around $s=0$, {as $s_{i-1}$ and $s_{i+1}$ tend to 0}, we obtain
				
				\begin{align}
					\bc_{i-1,r}(-s_{i-1}) &=
					\bc_{i-1,r}(0) - s_{i-1}\,\bc_{i-1,r}^{'}(0)+\frac{{s_{i-1}}^2}{2}\bc_{i-1,r}^{''}(0) - \frac{s_{i-1}^3}{6} \bc^{'''}_{i-1,r}(0)\label{eq:conicLTaylorparL} \\
					&\quad+O\left({s_{i-1}}^4\right), \nonumber \\
					\bc_{i-1,r}(s_{i+1}) &=
					\bc_{i-1,r}(0) + s_{i+1}\,\bc_{i-1,r}^{'}(0)+\frac{{s_{i+1}}^2}{2}\bc_{i-1,r}^{''}(0)
					+ \frac{{s_{i+1}}^3}{6}\bc^{'''}_{i-1,r}(0)  \label{eq:conicLTaylorparR} \\
					&\quad+O\left({s_{i+1}}^4\right),\nonumber\\
					\bc_{i,l}(-s_{i-1}) &= \bc_{i,l}(0) - s_{i-1}\,\bc_{i,l}^{'}(0)+\frac{{s_{i-1}}^2}{2}\bc_{i,l}^{''}(0)
					-\frac{s_{i-1}^3}{6} \bc^{'''}_{i,l}(0) + O\left({s_{i-1}}^4\right), \label{eq:conicRTaylorparL} \\
					\bc_{i,l}(s_{i+1}) &= \bc_{i,l}(0)+s_{i+1}\,\bc_{i,l}^{'}(0)+\frac{{s_{i+1}}^2}{2}\bc_{i,l}^{''}(0) + \frac{{s_{i+1}}^3}{6}\bc^{'''}_{i,l}(0) +O\left({s_{i+1}}^4\right). \label{eq:conicRTaylorparR}
				\end{align}
				
				From \eqref{eq:conicsparL} we deduce that $2\bc(-s_{i-1}) - \bc_{i-1,r}(-s_{i-1}) - \bc_{i,l}(-s_{i-1}) = 0$, which together with \eqref{eq:conicsZero}, \eqref{eq:conicsIderZero}, \eqref{eq:conicIIderparZero}, \eqref{eq:conicTaylorparL}, \eqref{eq:conicLTaylorparL}, and \eqref{eq:conicRTaylorparL}, {as $h, s_{i-1}$, and $s_{i+1}$ tend to 0}, leads to
				\begin{equation}\label{eq:TaylorparL}
					\begin{split}
						0 &= -\frac{{s_{i-1}}^3}{6} \left(2\,\bc^{'''}(0) -\bc_{i,l}^{'''}(0) -\bc_{i-1,r}^{'''}(0) \right) +2\,s_{{i}}O \left({h}^{4} \right) \bn_{{i}}\\
						&\quad + \frac{{s_{i-1}}^2}{2}\left(2\,\kappa(P_i) -\kappa_{i,l}(P_i) -\kappa_{i-1,r}(P_i) \right)\bn_{{i}}+{s_{{i}}}^{2}O \left({h}^{4} \right) \bt_{{i}} + O\left(s_{i-1}^4\right) .
					\end{split}
				\end{equation}
				Analogously, with \eqref{eq:conicsparR} on one side showing that $2\bc(s_{i+1}) - \bc_{i-1,r}(s_{i+1}) - \bc_{i,l}(s_{i+1}) = 0$, and substituting on the other side \eqref{eq:conicsZero}, \eqref{eq:conicsIderZero}, \eqref{eq:conicIIderparZero}, \eqref{eq:conicTaylorparR}, \eqref{eq:conicLTaylorparR}, and \eqref{eq:conicRTaylorparR}, {as $h, s_{i-1}$, and $s_{i+1}$ tend to 0}, we obtain
				\begin{equation}\label{eq:TaylorparR}
					\begin{split}
						0 &= \frac{{s_{i+1}}^3}{6} \left(2\,\bc^{'''}(0) -\bc_{i,l}^{'''}(0) -\bc_{i-1,r}^{'''}(0) \right) -2\,s_{{i+1}}O \left({h}^{4} \right) \bn_{{i}} \\
						&\quad+ \frac{{s_{i+1}}^2}{2}\left(2\,\kappa(P_i) -\kappa_{i,l}(P_i)-\kappa_{i-1,r}(P_i) \right)\bn_{{i}}+{s_{{i+1}}}^{2}O \left({h}^{4} \right) \bt_{{i}} + O\left(s_{i+1}^4\right) .
					\end{split}
				\end{equation}
				
				Simplifying the term $\left(2\,\bc^{'''}(0) -\bc_{i,l}^{'''}(0) -\bc_{i-1,r}^{'''}(0) \right)$ in \eqref{eq:TaylorparL} and \eqref{eq:TaylorparR}, and collecting coefficients of the normal vector we find
				\begin{equation}
					\frac{ 2\,\kappa(P_i) -\kappa_{i,l}(P_i) -\kappa_{i-1,r}(P_i)}{2} = 2\, \frac{s_{i-1}- s_{i+1}}{s_{i-1} s_{i+1}} O\left(h^4\right) . \label{Norma}
				\end{equation}
				Being $s_{i-1}=O(h)$ and $s_{i+1}=O(h)$, if $h \rightarrow 0$, it holds
				\begin{equation}
					\frac{s_{i-1}- s_{i+1}}{s_{i-1} s_{i+1}} O\left(h^4\right) =\, O(h^3) . \label{O3}
				\end{equation}
				
				Hence, after \eqref{Norma} and \eqref{O3} we get
				\begin{equation*}
					\kappa(P_i) -\frac{ \kappa_{i,l}(P_i) +\kappa_{i-1,r}(P_i)}{2}=\,O(h^3), \,\,{\text{as}\, s_{i-1}, s_{i+1}\, \rightarrow 0} .
				\end{equation*}
				
				Moreover, if the points $\{P_j, j=i-1,i,i+1\}$ are arc-length uniformly sampled, i.e., $s_{i-1}=s_{i+1}$, from \eqref{Norma} and \eqref{O3} it holds
				\begin{equation*}
					\kappa(P_i) -\frac{ \kappa_{i,l}(P_i) +\kappa_{i-1,r}(P_i)}{2}=\, O\left(h^4\right),\,\,{\text{as}\; s_{i-1}, s_{i+1}\rightarrow 0} . \hfill \qedsymbol
				\end{equation*}
			\end{proof}
			
			Recall that the general $3$-\emph{points} curvature approximations have first order approximation order. In the particular case when the three points are \textit{uniformly} arc-length distributed, the curvature approximation of the general $3$-\emph{points} curvature approximations is of second order \cite{Belyaev,Be99,Ca98}.
			
			Note that in addition to the third order approximation of the curvature estimator, which is obtained by averaging the curvatures at $P_i$ of the interpolating conics (numerical experiments in the next section show that in fact the approximation order may be very close to $O\left(h^4\right)$), this numerical approximation has another advantages. These include invariance with respect to the special Euclidean group of transformations $SE(2)$ and a low computational overhead using the output of ABFH.

			\section{Numerical experiments}
			\label{sec:num_ex}
			
			One important objective of the proposed method is to assign curvature values to a set of points, as an intermediate step for free design of curves, where only a small sample of unevenly distributed points are available and small variations in the curvature values do not appreciably affect the shape of the corresponding curve. \emph {ConicCurv} is not supposed to be used to approximate data from image processing or another applications, which are affected by noise or discretization errors. In \cite{Hermann07,Lewiner,Wo93} and in the references contained therein, there is a fairly complete study of the methods to estimate curvature in the case of digital spaces, i.e. curves extracted from images.

			\subsection{Comparison of curvature estimators}
			\label{subsec:comparisoncurvature}

			We tested the following four curvature estimation methods by measuring the relative error between the exact curvature value and the corresponding curvature estimates computed on the benchmark of representative curves given in \cite{Albrecht2} to test tangent estimations. The graph of the curves, the corresponding parametrization and the parametric values of the selected points are shown in Figures \ref{fig:pol}, \ref{fig:agnesi}, \ref{fig:folium}, \ref{fig:bicorn}, \ref{fig:teardrop}, \ref{fig:gaussian}, \ref{fig:ellipse}, and in Table \ref{tab:comp}.			
			
			Let $\bc(t)$ be the parametrization of any of the representative curves below. For five parametric values $t_j,\, j=1, \ldots,5$ the points $P_j=\bc(t_j)$ on $\bc(t)$ are computed and, depending on the curvature estimation method, either set of points $\{ P_j, \,j=2,3,4\}$ or $\{P_j, \,j=1,2,3,4,5\}$ is interpolated in order to obtain an estimation of the curvature at point $P_3=\bc(t_3)$.
			
			In our comparison we consider the following curvature estimation methods
			\begin{itemize}
				\item[$\triangleright$] \emph{Circle}: the estimate of curvature value at $P_3$ is the inverse of the radius of the circle interpolating $\{ P_j, \,j=2,3,4\}$.
				\item[$\triangleright$] \emph{Poly}$4$: the estimate of curvature value at $P_3$ is the curvature at $P_3$ of the fourth degree polynomial curve interpolating $\{P_j, \,j=1,2,3,4,5 \}$ at Chebyshev's nodes $\cos\left( \frac{(5 - 2j)\pi}{10} \right),\,j=-2,\ldots,2$.
				\item[$\triangleright$] \emph{Conic}: the estimate of curvature value at $P_3$ is the curvature of the conic interpolating $\{P_j, \,j=1,2,3,4,5 \}$.
				\item[$\triangleright$] \emph {ConicCurv}: the estimate of curvature value at $P_3$ is the average of the curvature values at $P_3$ of the conics interpolating the points $\{ P_j, \,j=2,3,4\}$ and the tangent directions assigned by means of ABFH to $\{ P_j, \,j=2,3\}$ and to $\{ P_j, \,j=3,4\}$, respectively.
			\end{itemize}
			
			The results from the tests are given in Table \ref{tab:rerror1}. We can observe that \emph {ConicCurv} shows a better performance than the standard curvature estimation methods in the case of planar convex data.
			
			\begin{table}[!hbt]
				\center
				\caption{Curves tested: curve type, parametrization, and parameter values.}\label{tab:comp}
				\begin{tabular}{|l|>{$}l<{$}|>{$}l<{$}|}  \hline
					\mbox{Curve} & \bc(t) =(\,x(t)\,,\, y(t)\,) & (t_1,\, t_2,\, t_3,\, t_4,\, t_5) \\ \thline
					\mbox{Polynomial} & \left( t\,,\,\frac15 -\frac15\, \left( 1-t \right) ^{5}\right) & (0,\,0.1,\,0.2,\,0.3,\,0.4 )     \\ \hline
					\mbox{Witch of Agnesi} & \left( t\,,\, \frac{t}{1+t^2} \right) & ( -2.25,\, -2,\, -1.5,\, -1,\, -0.75)     \\ \hline
					\mbox{Folium of Descartes} & \left( \,{\frac {3t}{{t}^{3}+1}}\,,\,{\frac {3{t}^{2}}{{t}^{3}+1}} \right) & (-0.1,\,0.1,\, 0.3,\, 0.5,\, 0.7)\\ \hline
					\mbox{Bicorn} & \left(\sin \left( t \right) \,,\,{\frac { {\cos}^{2} \left( t \right)}{2-\cos \left( t \right) }}\right) & ( 0.139,\, 0.278,\, 0.417,\, 0.556,\, 0.626 ) \\ \hline
					\mbox{Tear Drop} & \left( \cos \left( t \right) \,,\,\sin \left( t \right){\sin}^{2} \left( \frac{t}2 \right) \right) &( 1.867,\, 1.934,\, 2,\, 2.034,\, 2.067) \\ \hline
					\mbox{Exponential} & \left( t\,,\,{{\rm e}^{-2\, \left( t- 0.5 \right) ^{2}}} \right) & ( 0.2,\, 0.4,\, 0.5,\, 0.8,\, 0.9) \\ \hline
					\mbox{Ellipse} & \left( 5\cos(t)\,, \,2\sin(t)\right) & ( 0.539,\, 0.843,\, 1.222,\, 1.6,\, 1.904)  \\  \hline
				\end{tabular}
			\end{table}
			
			\begin{figure}[h!]
				\centering
				\subcaptionbox{Parameterization: $\left(t,\tfrac{1}{5}(1-(1-t)^5)\right)$}[.48\textwidth][c]{
					\begin{tikzpicture}
						\pgfplotsset{
							width=.5\textwidth,
							compat=newest,
							hide x axis, hide y axis}
						\begin{axis}[axis x line = center,axis y line = middle,
							xlabel={$x$},
							ylabel={$y$},
							xmin=0,xmax=1.7,ymin=0,ymax=.35,smooth,clip=true,
							xtick = {0.3,0.6,0.9,1.2,1.5},
							ytick = {0.1,0.2,0.3},
							ylabel style={yshift=.3cm,xshift=-.5cm},
							xlabel style={yshift=-.1cm,xshift=.5cm}
							]
							
							\addplot[domain=0:1.6,samples=200] (x,{.2 - .2*(1-x)^5});
							
							\shade[ball color=red!80!black] (.3,{.2 - .2*(1-.3)^5}) circle (2.1pt);
							\shade[ball color=blue!80!black] (.1,{.2 - .2*(1-.1)^5}) circle (2.1pt);
							\shade[ball color=blue!80!black] (.2,{.2 - .2*(1-.2)^5}) circle (2.1pt);
							\shade[ball color=blue!80!black] (.5,{.2 - .2*(1-.5)^5}) circle (2.1pt);
							\shade[ball color=blue!80!black] (.8,{.2 - .2*(1-.8)^5}) circle (2.1pt);
						\end{axis}
					\end{tikzpicture}
				}
				\hfill
				\subcaptionbox{Curvature graph and center point}[.48\textwidth][c]{
					\begin{tikzpicture}
						\pgfplotsset{
							width=.5\textwidth,
							compat=newest,
							hide x axis, hide y axis}
						\begin{axis}[axis x line = center,axis y line = middle,
							xlabel={$x$},
							ylabel={$y$},
							xmin=0,xmax=.8,ymin=0,ymax=1.8,smooth,clip=true,
							ylabel style={yshift=.3cm,xshift=-.5cm},
							xlabel style={yshift=-.1cm,xshift=.5cm},
							legend pos=north east]
							
							\addplot[domain=0:1.6,samples=200] (x,{(-4*(x-1)^3)/(1+(x-1)^8)^(3/2)});
							
							\shade[ball color=red!80!black] (axis cs:.3,1.2613) circle (2.1pt);				
						\end{axis}
					\end{tikzpicture}
				}
				\caption{Test curve: polynomial.} \label{fig:pol}
			\end{figure}
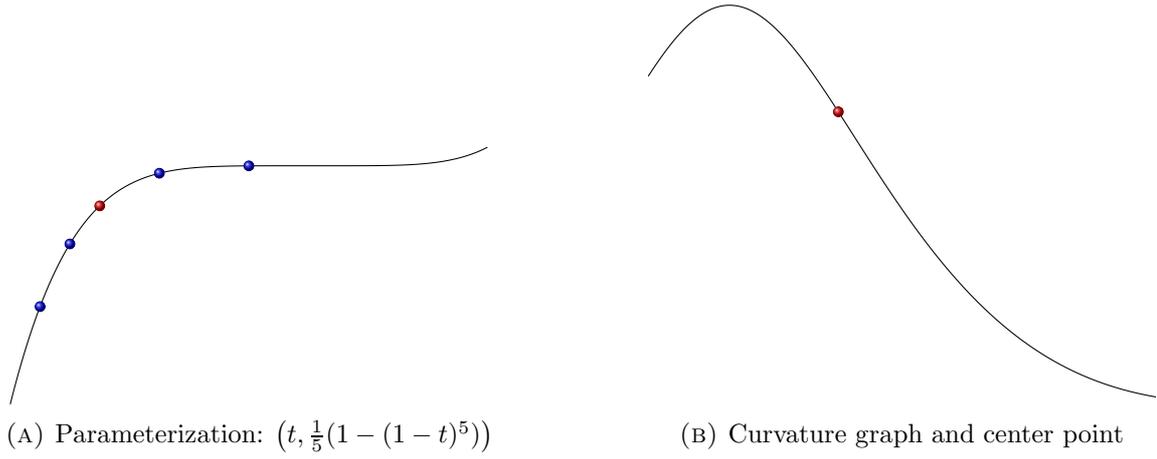
			
			\begin{figure}[h!]
				\centering
				\subcaptionbox{Parameterization: $\left( t, \frac{t}{1+t^2} \right)$ }[.48\textwidth][c]{
					\begin{tikzpicture}
						\pgfplotsset{
							width=.5\textwidth,
							hide x axis, hide y axis,
							compat=newest}
						\begin{axis}[axis x line = bottom, axis y line = left,
							xlabel={$x$},
							ylabel={$y$},
							xmin=-2.5,xmax=-0.4,
							ymin=-0.55,ymax=-0.31,
							smooth,clip=true,
							ylabel style={sloped like x axis,yshift=1.9cm,xshift=1.cm},
							xlabel style={yshift=1cm,xshift=2.55cm}
							]
							
							\addplot[domain=-2.5:-0.5,samples=200] ({x},{x/(1 + x^2)});
							
							\shade[ball color=red!80!black] ({-1.5},{(-1.5)/(1+1.5^2)}) circle (2.1pt);
							\shade[ball color=blue!80!black] ({-2.25},{(-2.25)/(1+2.25^2)}) circle (2.1pt);
							\shade[ball color=blue!80!black] ({-2.},{(-2.)/(1+2.^2)}) circle (2.1pt);
							\shade[ball color=blue!80!black] ({-1.},{(-1.)/(1+1.^2)}) circle (2.1pt);
							\shade[ball color=blue!80!black] ({-0.75},{(-0.75)/(1+0.75^2)}) circle (2.1pt);
						\end{axis}
					\end{tikzpicture}
				}
				\hfill
				\subcaptionbox{Curvature graph and center point}[.48\textwidth][c]{
					\begin{tikzpicture}
						\pgfplotsset{
							width=5cm,
							compat=newest,
							hide x axis, hide y axis
						}
						\begin{axis}[axis x line = center,axis y line = middle,
							xlabel={$x$},
							ylabel={$y$},
							xmin=-2.3,xmax=-0.4,ymin=-0.1,ymax=1.1,smooth,clip=true,
							ylabel style={yshift=.3cm,xshift=-.5cm},
							xlabel style={yshift=-.1cm,xshift=.5cm},
							legend pos=north east]
							
							\addplot[domain=-2.5:-0.5,samples=200,samples=200,variable=t] (t,{(2*((1+t^2)^3)*abs(t)*abs(-3+t^2))/((2 + 2*t^2 + 7*t^4 + 4*t^6 + t^8)^(3/2))});
							
							\shade[ball color=red!80!black] (axis cs:-1.5,0.0641907) circle (2.1pt);
							
						\end{axis}
					\end{tikzpicture}
				}
				\caption{Test curve: Witch of Agnesi.} \label{fig:agnesi}
			\end{figure}
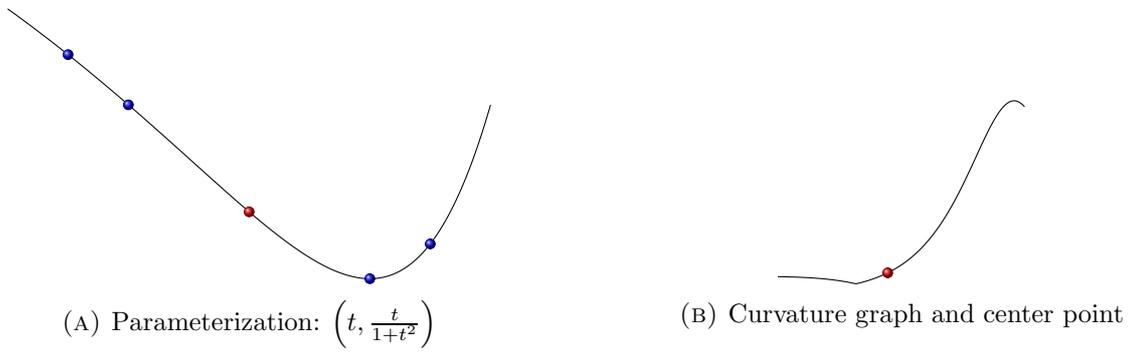
			
			\begin{figure}[h!]
				\centering
				\subcaptionbox{Parameterization: $\left( {\frac{3t}{{t}^{3}+1}}, {\frac{3{t}^{2}}{{t}^{3}+1}} \right)$}[.48\textwidth][c]{
					\begin{tikzpicture}
						\pgfplotsset{
							width=.5\textwidth,
							hide x axis, hide y axis,
							compat=newest}
						\begin{axis}[axis x line = center,axis y line = middle,
							xlabel={$x$},
							ylabel={$y$},
							xmin=0,xmax=1.8,ymin=0,ymax=1.7,
							smooth,clip=true,
							ylabel style={yshift=.3cm,xshift=-.5cm},
							xlabel style={yshift=-.1cm,xshift=.5cm}
							]
							
							\addplot[domain=0:2] ({3*x/(1+x^3)},{(3*x^2)/(1+x^3)});
							
							\shade[ball color=red!80!black] ({3*.5/(1+.5^3)},{(3*.5^2)/(1+.5^3)}) circle (2.1pt);
							\shade[ball color=blue!80!black] ({3*.1/(1+.1^3)},{(3*.1^2)/(1+.1^3)}) circle (2.1pt);
							\shade[ball color=blue!80!black] ({3*.25/(1+.25^3)},{(3*.25^2)/(1+.25^3)}) circle (2.1pt);
							\shade[ball color=blue!80!black] ({3*.85/(1+.85^3)},{(3*.85^2)/(1+.85^3)}) circle (2.1pt);
							\shade[ball color=blue!80!black] ({3*.9/(1+.9^3)},{(3*.9^2)/(1+.9^3)}) circle (2.1pt);
						\end{axis}
					\end{tikzpicture}
				}
				\hfill
				\subcaptionbox{Curvature graph and center point}[.48\textwidth][c]{
					\begin{tikzpicture}
						\pgfplotsset{
							width=.5\textwidth,
							compat=newest,
							hide x axis, hide y axis}
						\begin{axis}[axis x line = center,axis y line = middle,
							xlabel={$x$},
							ylabel={$y$},
							xmin=0,xmax=.7,ymin=.4,ymax=1.3,smooth,clip=true,
							xtick = {0.2,0.4,0.6},
							ylabel style={yshift=.3cm,xshift=-.5cm},
							xlabel style={yshift=-.1cm,xshift=.5cm},
							legend pos=north west]

							\addplot[domain=0:1.1,samples=200] {-((-36*x^2/(1+x^3)^2 + 54*x^5/(1+x^3)^3)*(6*x/(1+x^3)- 9*x^4/(1+x^3)^2) -(3/(1+x^3)-9*x^3/(1+x^3)^2)*(6/(1+x^3) -54*x^3/(1+x^3)^2 +54*x^6/(1+x^3)^3))/((3/(1+x^3) - 9*x^3/(1+x^3)^2)^2+(6*x/(1+x^3) -9*x^4/(1+x^3)^2)^2)^(3/2)};
							
							\shade[ball color=red!80!black] (axis cs:.5,.617) circle (2.1pt);
						\end{axis}
					\end{tikzpicture}
				}
				\caption{Test curve: Folium of Descartes.} \label{fig:folium}
			\end{figure}
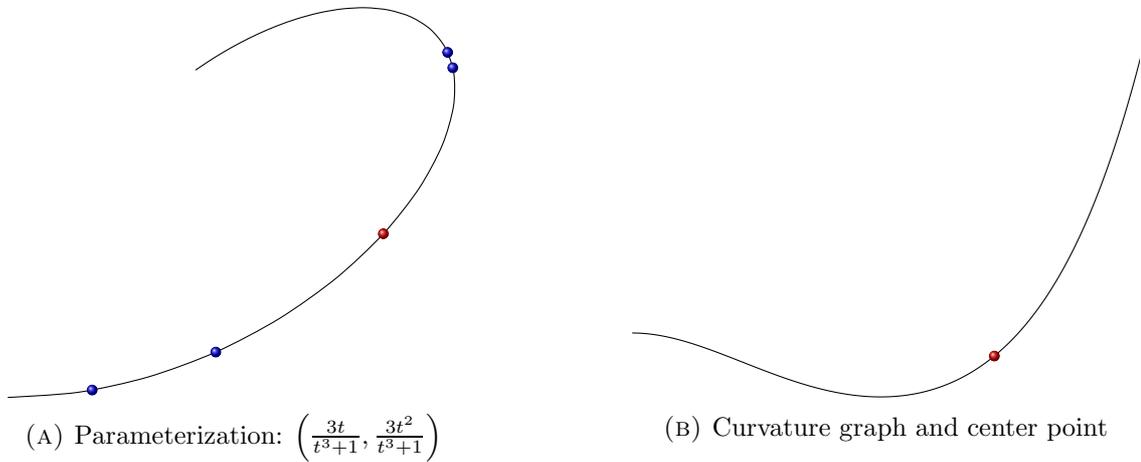
			
			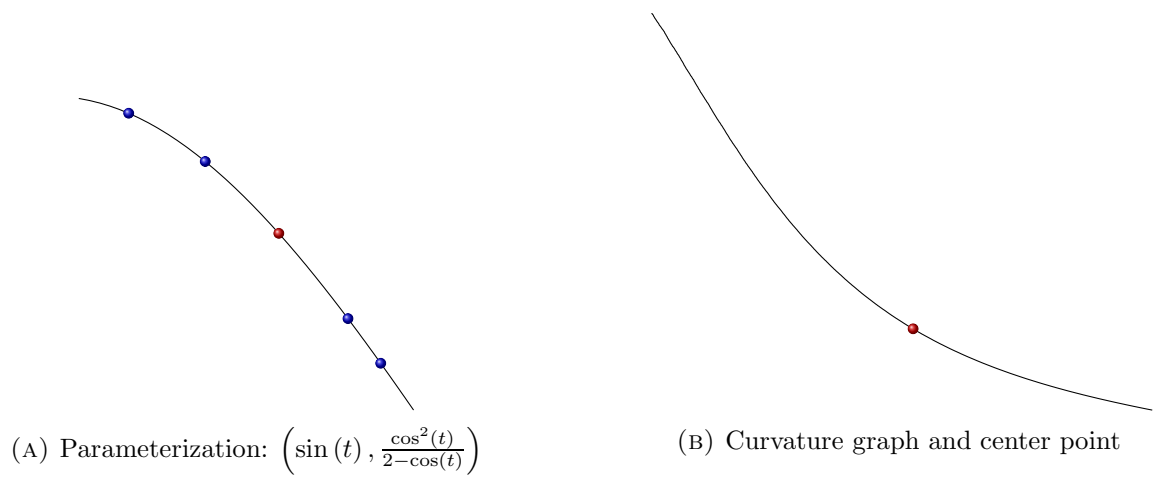
\begin{figure}[h!]
				\centering 	
				\subcaptionbox{Parameterization: $\left(\sin \left( t \right) ,{\frac { {\cos}^{2} \left( t \right)}{2-\cos \left( t \right) }}\right)$}[.48\textwidth][c]{
					\begin{tikzpicture}
						\pgfplotsset{
							width=.5\textwidth,
							hide x axis, hide y axis,
							compat=newest}
						\begin{axis}[axis x line = center,axis y line = middle,
							xlabel={$x$},
							ylabel={$y$},
							xmin=0,xmax=.9,ymin=0.4,ymax=1.1,smooth,clip=true,
							ylabel style={yshift=.3cm,xshift=-.5cm},
							xlabel style={yshift=-.1cm,xshift=.5cm}
							]
							
							\addplot[domain=0.05:0.7,samples=200,variable=t] ({sin(deg(t))},{(cos(deg(t))*cos(deg(t)))/(2 - cos(deg(t)))});
							
							\shade[ball color=red!80!black] (0.405019355426869,	0.769978416669331) circle (2.1pt);
							\shade[ball color=blue!80!black] (0.138552829041508,	0.971433678667396) circle (2.1pt);
							\shade[ball color=blue!80!black] (0.274432986257786,	0.890496996623145) circle (2.1pt);
							\shade[ball color=blue!80!black] (0.527792937030293,	0.626992640387948) circle (2.1pt);
							\shade[ball color=blue!80!black] (0.585907943376195,	0.55203389847862) circle (2.1pt);
						\end{axis}
					\end{tikzpicture}
				}
				\hfill
				\subcaptionbox{Curvature graph and center point}[.48\textwidth][c]{
					\begin{tikzpicture}
						\pgfplotsset{
							width=.5\textwidth,
							hide x axis, hide y axis,
							compat=newest}
						\begin{axis}[axis x line = center,axis y line = middle,
							xlabel={$x$},
							ylabel={$y$},
							xmin=0.1,xmax=0.7,ymin=0,ymax=2.51,
							smooth,clip=true,
							ylabel style={yshift=.3cm,xshift=-.5cm},
							xlabel style={yshift=-.1cm,xshift=.5cm}
							]
							
							\addplot[domain=0.1:0.7, samples=1000, variable=t] (t,{6*sqrt(2)*sqrt((((-2 + cos(deg(t)))^6)*((3 - 2*sec(deg(t)))^2))/((73 - 80*cos(deg(t)) + 9*cos(deg(2*t)))^3))});
							
							\shade[ball color=red!80!black] (0.417,0.615277) circle (2.1pt);
						\end{axis}
					\end{tikzpicture}
				}
				\caption{Test curve: Bicorn.} \label{fig:bicorn}
			\end{figure}
			
			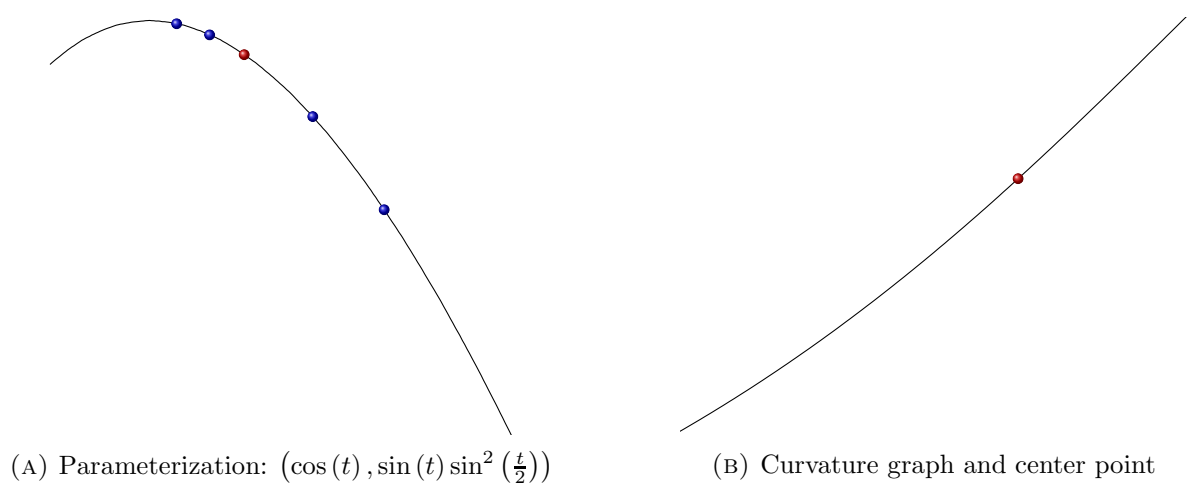
\begin{figure}[h!]
				\centering 	
				\subcaptionbox{Parameterization: $\left( \cos \left( t \right) ,\sin \left( t \right){\sin}^{2} \left( \frac{t}{2} \right) \right)$}[.48\textwidth][c]{
					\begin{tikzpicture}
						\pgfplotsset{
							width=.5\textwidth,
							hide x axis, hide y axis,	
							compat=newest}
						\begin{axis}[
							xmin=-0.6, xmax=-0.15,
							ymin=0.58, ymax=0.65,
							smooth,clip=true,
							]
							
							\addplot[domain=1.75:2.2,samples=210] ({cos(deg(x))},{sin(deg(x))*sin(deg(x)/2)*sin(deg(x)/2)});

							\shade[ball color=red!80!black] ({cos(deg(2))},{sin(deg(2))*sin(deg(2)/2)*sin(deg(2)/2)}) circle (2.1pt);
							\shade[ball color=blue!80!black] ({cos(deg(1.867))},{sin(deg(1.867))*sin(deg(1.867)/2)*sin(deg(1.867)/2)}) circle (2.1pt);
							\shade[ball color=blue!80!black] ({cos(deg(1.934))},{sin(deg(1.934))*sin(deg(1.934)/2)*sin(deg(1.934)/2)}) circle (2.1pt);
							\shade[ball color=blue!80!black] ({cos(deg(2.034))},{sin(deg(2.034))*sin(deg(2.034)/2)*sin(deg(2.034)/2)}) circle (2.1pt);
							\shade[ball color=blue!80!black] ({cos((deg(2.067)))},{sin(deg(2.067))*sin(deg(2.067)/2)*sin(deg(2.067)/2)}) circle (2.1pt);				 
						\end{axis}
					\end{tikzpicture}
				}
				\hfill
				\subcaptionbox{Curvature graph and center point}[.48\textwidth][c]{
					\begin{tikzpicture}
						\pgfplotsset{
							width=.5\textwidth,
							hide x axis, hide y axis,	
							compat=newest}
						\begin{axis}[
							xmin=1.8, xmax=2.1,
							ymin=0.75, ymax=1.75,
							smooth,clip=true,
							]
							
							\addplot[domain=1.75:2.2,samples=210,variable=t] {-sqrt(8)*(2*cos(deg(t)) + cos(deg(2*t)))*(cosec(deg(t)/2))/sqrt((5 + 4*cos(deg(t)) - cos(deg(3*t)))^3)};
							
							\shade[ball color=red!80!black] (2,1.36441) circle (2.1pt);				
						\end{axis}
					\end{tikzpicture}
				}
				\caption{Test curve: Tear Drop curve.} \label{fig:teardrop}
			\end{figure}
			
			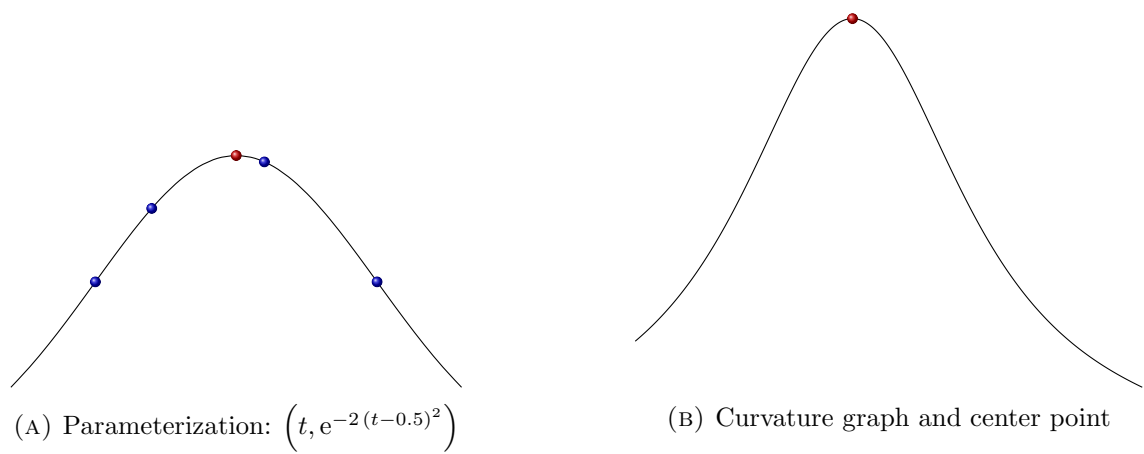
\begin{figure}[h!]
				\centering 	
				\subcaptionbox{Parameterization: $\left( t,{{\rm e}^{-2\, \left( t- 0.5 \right)^{2}}} \right)$ }[.48\textwidth][c]{
					\begin{tikzpicture}
						\pgfplotsset{
							width=.5\textwidth,
							hide x axis, hide y axis,	
							compat=newest}
						\begin{axis}[axis x line = center,axis y line = middle,
							xlabel={$x$},
							ylabel={$y$},
							xmin=-.4,xmax=1.4,ymin=0,ymax=1.3,smooth,clip=true,
							ylabel style={yshift=.3cm,xshift=-.5cm},
							xlabel style={yshift=-.1cm,xshift=.5cm}
							]
							\addplot[domain=-.3:1.3] {exp(-(sqrt(2)*x-(1/2)*sqrt(2))^2)};
							
							\shade[ball color=red!80!black] (.5,1) circle (2.1pt);
							\shade[ball color=blue!80!black] (0,0.6065) circle (2.1pt);
							\shade[ball color=blue!80!black] (.2,0.8353) circle (2.1pt);
							\shade[ball color=blue!80!black] (.6,0.9802) circle (2.1pt);
							\shade[ball color=blue!80!black] (1,0.6065) circle (2.1pt);
							
						\end{axis}
					\end{tikzpicture}
				}
				\hfill
				\subcaptionbox{Curvature graph and center point}[.48\textwidth][c]{
					\begin{tikzpicture}
						\pgfplotsset{
							width=.5\textwidth,
							hide x axis, hide y axis,
							compat=newest}
						\begin{axis}[axis x line = center,axis y line = middle,
							xlabel={$x$},
							ylabel={$y$},
							xmin=.2,xmax=0.9,ymin=0.1,ymax=4.3,smooth,clip=true,
							]
							
							\addplot[domain=0.2:0.9,samples=200] {-16*exp(-(1/2)*(2*x-1)^2)*x*(x-1)/(1+4*(2*x-1)^2*(exp(-(1/2)*(2*x-1)^2))^2)^(3/2)};
							
							\shade[ball color=red!80!black] (.5,4) circle (2.1pt);
							
						\end{axis}
					\end{tikzpicture}
				}
				\caption{Test curve: exponential.} \label{fig:gaussian}
			\end{figure}
			
			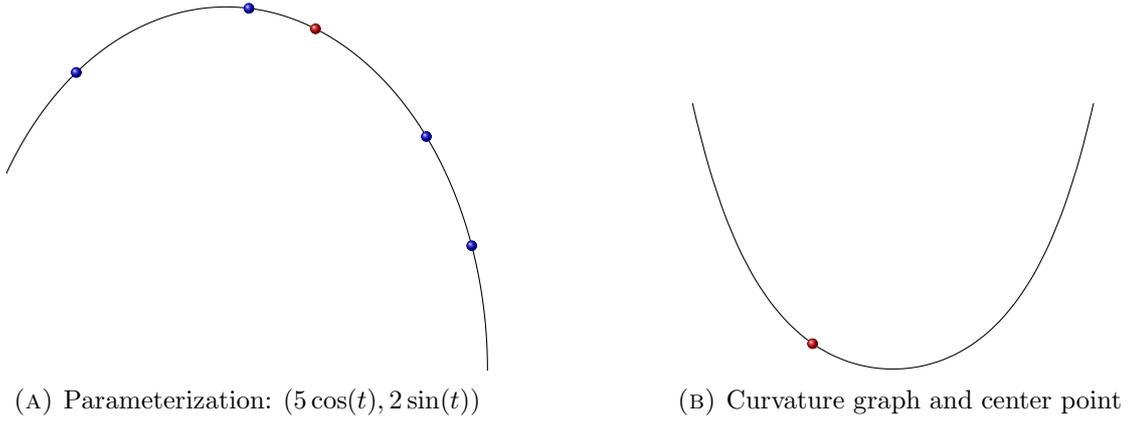
\begin{figure}[h!]
				\centering 	
				\subcaptionbox{Parameterization: $\left( 5\cos(t), 2\sin(t)\right)$ }[.48\textwidth][c]{
					\begin{tikzpicture}
						\pgfplotsset{
							width=.5\textwidth,
							hide x axis, hide y axis,	
							compat=newest}
						\begin{axis}[axis x line = center,axis y line = middle,
							xlabel={$x$},
							ylabel={$y$},
							xmin=-4.2,xmax=5.5,ymin=0,ymax=2.3,smooth,clip=true,
							ytick = {0.4,0.8,1.2,1.6,2},
							ylabel style={yshift=.3cm,xshift=-.5cm},
							xlabel style={yshift=-.1cm,xshift=.5cm}
							]
							
							\addplot[domain=-pi:pi,samples=250] ({5*cos(deg(x))},{2*sin(deg(x))});
							
							\shade[ball color=red!80!black] ({5*cos(deg(7*pi/18))},{2*sin(deg(7*pi/18))}) circle (2.1pt);
							\shade[ball color=blue!80!black] ({5*cos(deg(pi/9))},{2*sin(deg(pi/9))}) circle (2.1pt);
							\shade[ball color=blue!80!black] ({5*cos(deg(2*pi/9))},{2*sin(deg(2*pi/9))}) circle (2.1pt);
							\shade[ball color=blue!80!black] ({5*cos(deg(17*pi/36))},{2*sin(deg(17*pi/36))}) circle (2.1pt);
							\shade[ball color=blue!80!black] ({5*cos(deg(25*pi/36))},{2*sin(deg(25*pi/36))}) circle (2.1pt);
							
						\end{axis}
					\end{tikzpicture}
				}
				\hfill
				\subcaptionbox{Curvature graph and center point}[.48\textwidth][c]{
					\begin{tikzpicture}
						\pgfplotsset{
							width=.5\textwidth,
							hide x axis, hide y axis,
							compat=newest}
						\begin{axis}[axis x line = center,axis y line = middle,
							xlabel={$x$},
							ylabel={$y$},
							xmin=0.5,xmax=2.7,ymin=0,ymax=.22,smooth,clip=true,
							ytick = {-0.1,0.1,0.2,0.3},
							ylabel style={yshift=.3cm,xshift=-.5cm},
							xlabel style={yshift=-.1cm,xshift=.5cm}
							]

							\addplot[mark=none,samples=150] {-(-10*cos(deg(x))^2-10*sin(deg(x))^2)/(25*sin(deg(x))^2 + 4*cos(deg(x))^2)^(3/2)};
							
							\shade[ball color=red!80!black] (7*pi/18,.0934) circle (2.1pt);
							
						\end{axis}
					\end{tikzpicture}
				}
				\caption{Test curve: ellipse.} \label{fig:ellipse}
			\end{figure}

			\begin{table}[!hbt]
				\centering
				\caption{Relative error between the exact curvature value for the parameter value $t_3$ and the corresponding estimates.} \label{tab:rerror1}
				\begin{tabular}{|l|l|l|l|l|}  \hline
					\mbox{Curve} & \emph{Circle} & \mbox{\emph{Poly}4} & \mbox{\emph{Conic}} & \emph {ConicCurv}\\ \thline
					\mbox{Polynomial}     & 0.059 & 0.123 & 0.049  & 0.049  \\ \hline
					\mbox{Witch of Agnesi}   & 0.008 & 0.002 & 0.007  & 0.006  \\ \hline
					\mbox{Folium of Descartes} & 0.029 & 0.110 & 0.003  & 0.0001 \\ \hline
					\mbox{Bicorn} 				& 0.029 & 0.185 & 0.006  & 0.006  \\ \hline
					\mbox{Tear Drop}		  & 0.030 & 0.027 & 0.00008 & 0.00006 \\ \hline
					\mbox{Exponential}			& 0.245 & 0.790 & 0.017  & 0.0006 \\ \hline
					\mbox{Ellipse} 			& 0.326 & 0.032 & 0.000  & 0.000  \\ \hline
				\end{tabular}
			\end{table}

			\subsection{Approximation order: numerical experiments}
			\label{subsec:ApO}
			
			The previous numerical experiment shows that using a sample of unevenly distributed points, \emph{Conic} and \emph{ConicCurv} have a substantially better performance compared to another classical curvature estimation schemes.
			
			The next experiment focuses on the comparison of the approximation order of Conic and \emph{ConicCurv}. Recall that both methods are $SE(2)$ invariant, have \emph{conic precision}, and depend on five data (either $5$ points or $3$ points and $2$ tangent directions).
			
			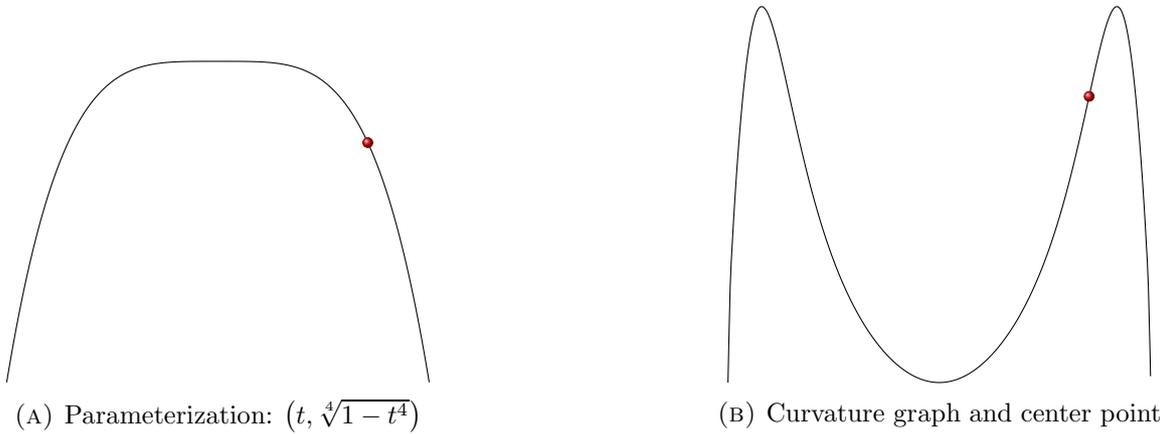
\begin{figure}[h!]
				\centering 	
				\subcaptionbox{Parameterization: $\left(t, \sqrt[4]{1-t^4}\right)$ }[.4\textwidth][c]{
					\begin{tikzpicture}
						\pgfplotsset{
							width=.5\textwidth,
							hide x axis, hide y axis,
							compat=newest}
						\begin{axis}[axis x line = center,axis y line = middle,
							xlabel={$x$},
							ylabel={$y$},
							xmin=-1.2,xmax=1.2,ymin=-0.2,ymax=1.1,smooth,clip=true,
							ylabel style={yshift=.3cm,xshift=-.5cm},
							xlabel style={yshift=-.1cm,xshift=.5cm}
							]
							
							\addplot[domain=-1.0:1.0,samples=200] ({x},{sqrt{sqrt{1-x^4}}});
							
							\shade[ball color=red!80!black] (.7093,{sqrt{sqrt{1-.7093^4}}}) circle (2.1pt);
							
						\end{axis}
					\end{tikzpicture}
				}
				\hfill
				\subcaptionbox{Curvature graph and center point}[.45\textwidth][c]{
					\begin{tikzpicture}
						\pgfplotsset{
							width=.5\textwidth,
							hide x axis, hide y axis,
							compat=newest}
						\begin{axis}[axis x line = center,axis y line = middle,
							xlabel={$x$},
							ylabel={$y$},
							xmin=-1.2,xmax=1.2,ymin=-0.2,ymax=2.6,smooth,clip=true,
							ylabel style={yshift=.3cm,xshift=-.5cm},
							xlabel style={yshift=-.1cm,xshift=.5cm}
							]
							
							\addplot[domain=-1.0:1.0,samples=200, variable=t] {3*(t^2)*sqrt(((1 - t^4)^(5/2))/((1 - 2*t^4 + t^8 + (t^6)*sqrt(1 - t^4))^3))};
							
							\shade[ball color=red!80!black] (.7093,1.91995) circle (2.1pt);
							
						\end{axis}
					\end{tikzpicture}
				}
				\caption{Graph of chosen function to compare the methods: Conic and {\em ConicCurv}.
				}
				\label{fig:curva-quart}
			\end{figure}
			
			Let $f: \, I^{k} \rightarrow \mathbb{R}$ be the $C^3$-continuous function \[f(t)=\sqrt[4]{1-t^4}\]
			in a neighborhood of $t_3=0.7093 $ and let $\{ h^{k}={\frac {0.4}{\sqrt {k+2}}},\,k=0,1,2, \ldots,7 \}$ be	a strict monotonic decreasing sequence. The exact value of the curvature of $\bc(t)=(t,f(t))$ at $P_3=(t_3,f(t_3))=(0.7093, 0.9296)$ is $\kappa_{ex}=1.9199$; see Fig. \ref{fig:curva-quart}.
			
			We estimate the approximation order of the curvature values at $t_3$ obtained with 2 methods, \emph{Conic} and \emph{ConicCurv}, when the values of $t_i$ are selected in the interval $$I^{k}=[0.7093-h^{k}\,,\,0.7093+h^{k}],$$ for $k=0,1,2, \ldots,7$, and $t_3^k=t_3$; see Fig. \ref{fig:ConicExper}.

			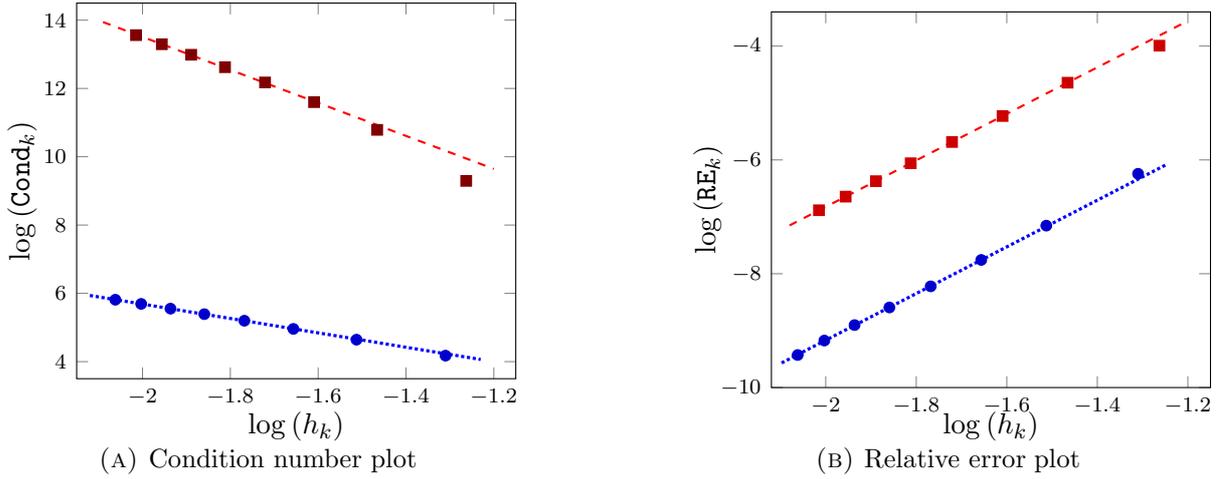
\begin{figure}[h]
				\centering	
				\captionsetup[subfigure]{aboveskip=-2pt,belowskip=-2pt}
				\subcaptionbox{Condition number plot \label{subfig:condExp}}[.45\textwidth][c]{
					\begin{tikzpicture}
						\pgfplotsset{compat=newest,
							label shift=-.03*.35\textwidth,
							tick label style={font=\scriptsize},
							legend style={font=\footnotesize},
						}
						\begin{axis}[scale only axis,
							width=.35\textwidth,
							xmin=-2.15, xmax=-1.15,
							xminorticks=true,
							xlabel={$ \log{(h_k)}$},
							ymin=3.5, ymax=14.5,
							yminorticks=true,
							ytick distance=2,
							ylabel={$ \log{(\texttt{Cond}_k)}$},
							ylabel near ticks]			
							\addplot[domain=-2.09:-1.2,variable=h,color=red,dashed, thick] {3.84619541425311 - 4.83154749747642*h};
							\addplot [color=red!50!black, only marks, mark=square*, mark options={solid}, forget plot]
							table[row sep=crcr]{
								-1.2629  9.2927 \\
								-1.4656  10.7870 \\
								-1.6094  11.5986 \\
								-1.7210  12.1754 \\
								-1.8122  12.6219 \\
								-1.8892  12.9864 \\
								-1.9560  13.2943 \\
								-2.0149  13.5610 \\
							};
							
							\addplot[domain=-2.12:-1.23,variable=h,color=blue,densely dotted,very thick] {1.47318285060851 - 2.10574069427154*h};
							\addplot [color=blue!80!black, only marks, mark=*, mark options={solid}, forget plot]
							table[row sep=crcr]{
								-1.3100  4.1754 \\
								-1.5127  4.6430 \\
								-1.6565  4.9587 \\
								-1.7681  5.1979 \\
								-1.8593  5.3906 \\
								-1.9363  5.5520 \\
								-2.0031  5.6909 \\
								-2.0620  5.8129 \\
							};

						\end{axis}
					\end{tikzpicture}	
				}
				\hfill
				\subcaptionbox{Relative error plot \label{subfig:REExp}}[.45\textwidth][c]{
					\begin{tikzpicture}			
						\pgfplotsset{compat=newest,
							label shift=-.05*.35\textwidth,
							tick label style={font=\scriptsize},
							legend style={font=\footnotesize},
						}	
						\begin{axis}[scale only axis,
							width=.35\textwidth,
							xmin=-2.12, xmax=-1.15,
							xminorticks=true,
							xlabel={$\log{(h_k)}$},
							ymin=-10, ymax=-3.4,
							yminorticks=true,				
							ytick distance=2,
							ylabel={$\log{(\texttt{RE}_k)}$},
							ylabel near ticks]				
							\addplot[domain=-2.08:-1.2,variable=h,color=red,dashed, thick] {1.34352208918068 + 4.08406612082821*h};
							\addplot [color=red!80!black, only marks, mark=square*, mark options={solid}, forget plot]
							table[row sep=crcr]{
								-1.2629  -3.9929 \\
								-1.4656  -4.6433 \\
								-1.6094  -5.2286 \\
								-1.7210  -5.6856 \\
								-1.8122  -6.0584 \\
								-1.8892  -6.3730 \\
								-1.9560  -6.6450 \\
								-2.0149  -6.8844 \\
							};
							
							\addplot[domain=-1.25:-2.1,variable=h,color=blue,densely dotted,very thick] {-0.968031470809855 + 4.09982387298344*h};
							\addplot [color=blue!80!black, only marks, mark=*, mark options={solid}, forget plot]
							table[row sep=crcr]{
								-1.3100  -6.2434 \\
								-1.5127  -7.1549 \\
								-1.6565  -7.7575 \\
								-1.7681  -8.2215 \\
								-1.8593  -8.5927 \\
								-1.9363  -8.9021 \\
								-2.0031  -9.1741 \\
								-2.0620  -9.4281 \\
							};
						\end{axis}			
					\end{tikzpicture}	
				}
				\caption{Experiments with the methods {\em Conic} (square red data) and \emph{ConicCurv} (circle blue data). (a) Order of condition number for \emph{Conic} $\approx O(h^{-5})$ and for \emph{ConicCurv} $\approx O(h^{-2})$. (b) Convergence order of \emph{Conic} $\approx O(h^{4})$ and for \emph{ConicCurv} $\approx O(h^{4})$.}
				\label{fig:ConicExper}
			\end{figure}
			
			\paragraph{\textbf{\emph{Conic}}}
			For $k=0,1,2, \ldots,7$, it is computed the implicit equation of the conic interpolating the 5 points $\{P_j^k=(t_j^k ,f(t_j^k )), j=1,\ldots,5\}$, with $t_j^k \in I^{k}$. The condition number of the $5\times 5$ matrix of the linear system of equations whose solution are the coefficients of the implicit equation of the interpolating conic is stored in $\texttt{Cond}_k$ and the relative error between the curvature at $t_3^k$ of the interpolating conic and the exact curvature value is assigned to $\texttt{RE}_k$.
			
			The line fitting the points $(\log(h_k),\log(\texttt{Cond}_k))$ for $k \geq 2$ has the slope $-4.831$, hence the order of the condition number of the $5\times 5$ matrix is close to $O(h^{-5})$; see Figure \ref{subfig:condExp} (red data: square points and dashed lines).
			
			The line fitting the points $(\log(h_k), \log(\texttt{RE}_k))$ for $k \geq 2$ has the slope $4.084$. Therefore, the approximation order of the curvature estimation with \emph{Conic} is close to $O(h^{4})$. See Figure \ref{subfig:REExp} (red data: square points and dashed lines).

			\paragraph{\textbf{\emph{ConicCurv}}}
			For $k=0,1, \ldots,7$, seven points $\{P_j^k=(t_j^k ,f(t_j^k )), j=1,\ldots,7\}$, with $t_j^k \in I^{k}$ are computed. Tangent lines are assigned to the points $\{P_j^k, j=2,3,4\}$, by means of ABFH applied to the points $P^k_{j-2}$, $P^k_{j-1}$, $P^k_{j}$, $P^k_{j+1}$, $P^k_{j+2}$ for $j=3,4,5$, which requires the solution of seven $2 \times 2$ linear systems of equations.
			Then, the auxiliary points $Q^k_{34}= r_3^k \bigwedge r_4^k $ and $Q^k_{45}=r_4^k \bigwedge r_5^k $ are computed as the solutions of two $2 \times 2$ linear systems of equations. The maximum of the condition numbers of the previous nine $2\times 2$ matrices is assigned to $ \texttt{Cond}_k$.
			
			The line fitting the points $(\log(h_k),\log(\texttt{Cond}_k))$ for $k \geq 2$ has the slope $ -2.105 $, hence the order of the maximum of the condition numbers of the $2\times 2$ matrices necessary to compute the curvature estimate by means of \emph{ConicCurv} is close to $O(h^{-2})$; see Figure \ref{subfig:condExp} (blue data: circle points and dotted lines).
			
			The estimate of curvature value at $P_3$ is the average of the curvature values at $P_3$ of the conics interpolating the points $\{ P_j^k, \,j=2,3,4\}$ and the tangent directions assigned by means of ABFH to $\{ P_j^k, \,j=2,3\}$ and to $\{ P_j^k, \,j=3,4\}$, respectively. The relative error of the curvature value at $P_3$ estimated with \emph{ConicCurv} and the exact value is assigned to $ \texttt{RE}_k$.
			
			The line fitting the points $(\log(h_k), \log(\texttt{RE}_k))$ for $k \geq 2$ has the slope $4.09.$ Thus, the approximation order of the curvature estimation with \emph{ConicCurv} is close to $O(h^{4})$.; see Figure \ref{subfig:REExp} (blue data: circle points and dotted lines).

			\begin{remark}
				\begin{itemize}
					\item[$\triangleright$] In a neighborhood of point $P_3$ the parametrization $\bc(t)=(t,f(t))$ behaves very close to the arc-length parametrization.
					\item[$\triangleright$] The computational overhead and the numerical condition of \emph{Conic} are mainly determined by the solution of the linear system of equations of size $5 \times 5$, that provides the coefficients of the implicit equation of the interpolating conic. The computational overhead and the numerical condition of \emph{ConicCurv} are mainly given by the solution of nine linear systems of equations of size $2 \times 2$.
					\item[$\triangleright$] The relative error of the curvature value at $P_3$ estimated with \emph {ConicCurv} and the exact value is smaller than the relative error between the curvature at $t_3^k$ of the interpolating conic and the exact curvature value.
				\end{itemize}
			\end{remark}

			\section{Application to corner estimation of $L$-curves}
			\label{sec:Lcurv}
			
			There are applications where the given set of points has no geometrical meaning for a curve design or a curve approximation, instead they represent data obtained from a particular problem, such as the residual and solution norms of computed solutions dependent on a parameter. The choice of such a parameter is crucial to obtain accurate solutions to ill-posed problems.
			
			Let us consider for instance the linear least-squares problem
			\begin{equation*} \label{eq:lsq}
				\min_x \; \lVert Ax - b \rVert_2^2.
			\end{equation*}
			When the matrix $ A $ is ill-conditioned, i.e, its singular values decay rapidly to zero without a significant gap, it is an \emph{ill-posed problem} for which the solution $x = A^\dagger b$ is not stable, where $A^\dagger$ is the Moore-Penrose pseudo-inverse. One of the best-known solution strategies is the Tikhonov regularization, defined by
			\begin{equation} \label{eq:reg_Tij_g}
				\min_{x} \; \lVert Ax - b \rVert_2^2 + \alpha^2 \lVert L(x - x_0) \rVert_2^2,
			\end{equation}
			where $\alpha$ is the regularization parameter chosen by the user. The vector $x_0$ is an {\em a priori} estimation of the solution. If $x_0$ is unknown, then it can be set equal the null vector. The regularization \eqref {eq:reg_Tij_g}, in its standard form with $L = I$ (identity matrix), reduces to
			\begin{equation} \label{eq:reg_Tij}
				\min_{x} \; \lVert Ax - b \rVert_2^2 + \alpha^2 \lVert x \rVert_2^2.
			\end{equation}
			
			There are several strategies to locate the optimal value of the parameter $\alpha$, one of them is to find the point of maximum curvature (corresponding to the best regularization parameter) of the parametric graph of the norm of the solution versus the corresponding residual norm.
			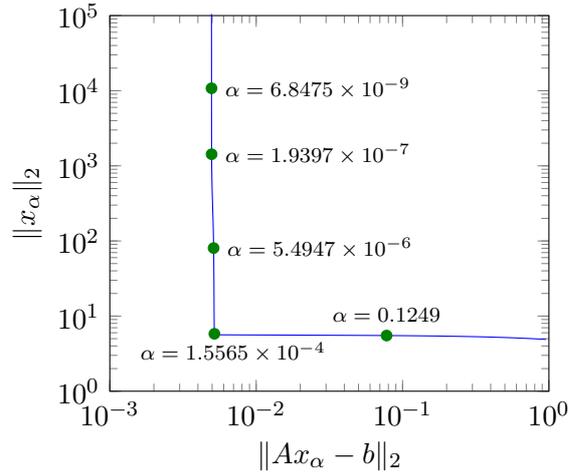
\begin{figure}
				\centering
				\begin{tikzpicture}
					\pgfplotsset{compat=newest,
						label shift=-.01*.35\textwidth,
					}	
					\begin{axis}[%
						width=.35\textwidth,
						clip=false,
						scale only axis,
						xmode=log,
						xmin=0.001, xmax=1,
						xminorticks=true,
						xlabel={$\lVert A x_\alpha - b \rVert_2$},
						ymode=log,
						ymin=1, ymax=100000,
						yminorticks=true,
						ylabel={$\lVert x_\alpha \rVert_2$}
						]
						
						\addplot [color=blue, solid, forget plot]
						table[row sep=crcr]{
							0.954126597890215 4.89176218099661\\
							0.736499267840419 5.02913004821045\\
							0.564024979076664 5.14317127117825\\
							0.428936137028368 5.2365009446436\\
							0.324129255931391 5.31196150793536\\
							0.243576957724415 5.3722423777672\\
							0.182344949692846 5.41977086858788\\
							0.136414475431614 5.45674178928488\\
							0.1024859502335 5.48515705848457\\
							0.0778273227964999 5.50683100201845\\
							0.0601641012440697 5.5233751201958\\
							0.0476047417620769 5.5361840719849\\
							0.0386076592673861 5.54642993243924\\
							0.0319810351913418 5.55506172237289\\
							0.0268770011530711 5.56280896062493\\
							0.0227478493934536 5.57019246073153\\
							0.0192717784415925 5.57754243312906\\
							0.0162778562859894 5.58501703233151\\
							0.0136886755903876 5.59261727683399\\
							0.011480598599605 5.60020718289301\\
							0.0096543979965732 5.60755412163217\\
							0.00821158201065356 5.61439262897333\\
							0.00713636159064419 5.62049471956448\\
							0.00638672259602757 5.62572066744732\\
							0.00589827136343519 5.63003462069004\\
							0.00559871799382051 5.63348870735531\\
							0.00542353532297301 5.6361913620992\\
							0.00532431461811625 5.63827520474563\\
							0.00526905384200871 5.63987276665095\\
							0.00523837145315345 5.6411019083551\\
							0.00522118462604428 5.64205969444633\\
							0.00521136434599525 5.64282275137379\\
							0.00520556690371205 5.64345180295045\\
							0.00520196956615163 5.64399748638643\\
							0.0051995810233987 5.64450438420877\\
							0.00519787391583046 5.64501125145091\\
							0.00519658151315699 5.64554779391097\\
							0.00519557533860351 5.64613139852468\\
							0.0051947880018116 5.64676922769758\\
							0.00519416870372807 5.64747025758269\\
							0.00519366430525251 5.64826906606907\\
							0.00519321648847034 5.64926274115103\\
							0.00519276478127086 5.65066787505583\\
							0.00519224877494546 5.65291613904442\\
							0.00519160769795254 5.65682238894038\\
							0.00519077860912823 5.66387677331709\\
							0.00518969554801417 5.67672614618803\\
							0.00518829210671749 5.69989932994501\\
							0.00518650971964987 5.74074447428797\\
							0.00518431312812319 5.81030297986044\\
							0.00518171183593991 5.92340093364351\\
							0.0051787811184541 6.09681636185659\\
							0.00517566994311733 6.34471267668501\\
							0.0051725819431336 6.67242465310315\\
							0.0051697263775772 7.07236964027257\\
							0.00516725660381201 7.52635366544179\\
							0.00516522692323339 8.01523815739177\\
							0.00516358845182074 8.53296980254347\\
							0.00516221833473338 9.10145602978878\\
							0.00516095897924498 9.78514426774598\\
							0.00515964718207367 10.705131132486\\
							0.005158127002864 12.0492079809947\\
							0.00515625057183171 14.0701019724741\\
							0.00515387311198033 17.0703452390982\\
							0.00515084547778628 21.3899018914894\\
							0.00514700405927156 27.4184004410637\\
							0.00514215762941149 35.635292910574\\
							0.00513607569725605 46.6597373374194\\
							0.00512849143006062 61.2803132019322\\
							0.00511913742571156 80.4296259912022\\
							0.0051078268461492 105.074333113161\\
							0.00509457290229496 136.011679685268\\
							0.00507970941444078 173.597905437746\\
							0.00506394474748038 217.482409294412\\
							0.00504827716015405 266.474962455627\\
							0.00503375450484814 318.688079746364\\
							0.00502116643211154 372.018201913965\\
							0.00501082868845679 424.868691850106\\
							0.00500257118893329 476.87868309656\\
							0.00499590731371817 529.394770801363\\
							0.00499026716740764 585.470745780444\\
							0.0049851866185175 649.226232535192\\
							0.00498040686118266 724.520967188765\\
							0.00497588599501611 813.281942788557\\
							0.0049717405099246 914.283587937844\\
							0.00496814424863756 1023.15717165674\\
							0.00496522649259012 1133.72363613257\\
							0.00496301228177653 1239.99672838353\\
							0.0049614224085809 1337.99640355571\\
							0.00496031484639566 1426.87592859354\\
							0.00495953380029576 1509.40683651922\\
							0.00495894327296009 1592.15592004489\\
							0.00495844099773453 1685.52577815608\\
							0.00495795967051589 1803.28478747341\\
							0.00495746343338765 1960.67210610386\\
							0.00495694367201812 2170.43360031096\\
							0.00495641406436815 2437.76882869879\\
							0.0049559029431188 2756.92474825825\\
							0.0049554423335814 3111.83938520154\\
							0.00495505638606445 3480.9113773242\\
							0.00495475397361786 3843.91513876541\\
							0.00495452830684499 4188.66750474551\\
							0.00495436208711181 4516.14276576702\\
							0.00495423411518361 4844.41739917289\\
							0.0049541240946057 5212.79438708549\\
							0.00495401476490351 5686.5509483657\\
							0.0049538921927532 6359.68223237294\\
							0.00495374544933193 7349.97784745324\\
							0.00495356655004405 8783.2583067941\\
							0.00495335096533373 10772.6721799079\\
							0.00495309840424193 13403.8702060378\\
							0.00495281299193708 16732.7962273442\\
							0.00495250170028978 20801.6316254334\\
							0.00495217043542621 25683.6538867393\\
							0.00495181858629981 31567.4514806202\\
							0.00495143407595693 38876.5189508154\\
							0.00495099067362584 48396.4490699533\\
							0.00495044759265969 61367.5368123089\\
							0.00494974991769249 79521.3859690701\\
							0.00494882848610869 105092.870002254\\
						};
						
						\addplot [color=green!50!black, only marks, mark=*, mark options={solid}, forget plot]
						table[row sep=crcr]{
							0.0778273227964999 5.50683100201845\\
							0.00518431312812319 5.81030297986044\\
							0.00511913742571156 80.4296259912022\\
							0.00496031484639566 1426.87592859354\\
							0.00495335096533373 10772.6721799079\\
						};
						
						\node[above=5pt, inner sep=0mm, text=black]
						at (axis cs:0.0778273227964999,5.50683100201845,0) {\scriptsize $\alpha = 0.1249$};
						\node[right=7pt,below=3pt, inner sep=0mm, text=black]
						at (axis cs:0.00518431312812319,5.81030297986044,0) {\scriptsize $\alpha = 1.5565\times10^{-4}$};
						\node[right=5pt, inner sep=0mm, text=black]
						at (axis cs:0.00511913742571156,80.4296259912022,0) {\scriptsize $\alpha = 5.4947\times10^{-6}$};
						\node[right=5pt, inner sep=0mm, text=black]
						at (axis cs:0.00496031484639566,1426.87592859354,0) {\scriptsize $\alpha = 1.9397\times10^{-7}$};
						\node[right=5pt, inner sep=0mm, text=black]
						at (axis cs:0.00495335096533373,10772.6721799079,0) {\scriptsize $\alpha = 6.8475\times10^{-9}$};
					\end{axis}
				\end{tikzpicture}%
				\caption{Graph of a $L$-curve and some parametric values.}
				\label{fig:Lcurv}
			\end{figure}
			
			The shape of this curve suggests the letter $L$ (see Fig. \ref{fig:Lcurv}), therefore, it is called the $L$ -curve. It is customary to process the $L$-curve in $\boldsymbol{\log}$-$\boldsymbol{\log}$ scale
			\begin{equation} \label{eq:lcurve}
				c(\alpha) = \left(\log\lVert Ax_\alpha - b \rVert_2,\log\lVert x_\alpha \rVert_2\right)
			\end{equation}
			for the calculation of the optimal parameter \cite{Hansen_pruning,Hansen}.
			
			Since the evaluation of points on the $L$-curve is computationally expensive, it is then preferable to estimate the corner point from a small sample of points on the $L$-curve. In \cite{CMCR00,CGG02,KiRa19} are proposed algorithms to estimate the location of the corner point of a L-curve. In this context, if \emph{ConicCurv} is applied to assign curvature values to the sample points and to find among them the one corresponding to the maximum curvature, then it happens to be a successful corner point estimation method. Recall that in this case \emph{ConicCurv} furnishes us a \emph{derivative-free} $L$-curve corner location method, since we do not need to estimate derivatives as in \cite{Hansen}.
			
			In order to validate this proposal, numerical experiments are performed with the test problems of the MATLAB package \textbf{regutool} (Regularization Tools) \cite{Regutool}. For each problem, several samples of points were selected on the respective $L$-curve \eqref{eq:lcurve}, considering such points in $\boldsymbol{\log}$-$\boldsymbol{\log}$ scale. The obtained results show the reliability of the algorithm based on \emph{ConicCurv} to obtain the point of maximum curvature and thus, the optimal parameter for regularization. For simplicity, only results associated with two test problems are exposed below: \texttt{shaw(32)} and \texttt{heat(64)}, representatives of problems identified as easy and problematic, respectively \cite{Hansen_pruning}. From these two test problems we take the matrix $A$ and the vector $b$ of \eqref{eq:lcurve}, being $x_\alpha$ solution of \eqref{eq:reg_Tij}.
			
			In fact, for the points shown in Fig. \ref{subfig:shaw_exp}, selected from the $L$-curve of the problem \texttt{shaw(32)} the estimated values are obtained with \emph{ConicCurv} and shown in Table \ref{tab:Lcurv}. If such estimates are compared with Fig. \ref{subfig:shaw_exp}, the correctness of selecting the point with label $6$ as corner point is verified.

			\begin{figure}
				\centering
				\subcaptionbox{\texttt{shaw(32)} \label{subfig:shaw_exp}}[.45\textwidth][c]{
					\begin{tikzpicture}	
						\pgfplotsset{compat=newest,
							label shift=-.01*.35\textwidth,
						}	
						\begin{axis}[%
							width=.35\textwidth,
							scale only axis,
							xmode=log,
							xmin=0.001,
							xmax=1.1,
							xminorticks=true,
							xlabel={$\lVert A x_\alpha - b \rVert_2$},
							ymode=log,
							ymin=1,
							ymax=100000,
							yminorticks=true,
							ylabel={$\lVert x_\alpha \rVert_2$}
							]
							
							\addplot [color=blue, solid, forget plot]
							table[row sep=crcr]{
								1.2245577097565 4.7292820134482\\
								0.954126597890215 4.89176218099661\\
								0.736499267840419 5.02913004821045\\
								0.564024979076664 5.14317127117825\\
								0.428936137028368 5.2365009446436\\
								0.324129255931391 5.31196150793536\\
								0.243576957724415 5.3722423777672\\
								0.182344949692846 5.41977086858788\\
								0.136414475431614 5.45674178928488\\
								0.1024859502335 5.48515705848457\\
								0.0778273227964999 5.50683100201845\\
								0.0601641012440697 5.5233751201958\\
								0.0476047417620769 5.5361840719849\\
								0.0386076592673861 5.54642993243924\\
								0.0319810351913418 5.55506172237289\\
								0.0268770011530711 5.56280896062493\\
								0.0227478493934536 5.57019246073153\\
								0.0192717784415925 5.57754243312906\\
								0.0162778562859894 5.58501703233151\\
								0.0136886755903876 5.59261727683399\\
								0.011480598599605 5.60020718289301\\
								0.0096543979965732 5.60755412163217\\
								0.00821158201065356 5.61439262897333\\
								0.00713636159064419 5.62049471956448\\
								0.00638672259602757 5.62572066744732\\
								0.00589827136343519 5.63003462069004\\
								0.00559871799382051 5.63348870735531\\
								0.00542353532297301 5.6361913620992\\
								0.00532431461811625 5.63827520474563\\
								0.00526905384200871 5.63987276665095\\
								0.00523837145315345 5.6411019083551\\
								0.00522118462604428 5.64205969444633\\
								0.00521136434599525 5.64282275137379\\
								0.00520556690371205 5.64345180295045\\
								0.00520196956615163 5.64399748638643\\
								0.0051995810233987 5.64450438420877\\
								0.00519787391583046 5.64501125145091\\
								0.00519658151315699 5.64554779391097\\
								0.00519557533860351 5.64613139852468\\
								0.0051947880018116 5.64676922769758\\
								0.00519416870372807 5.64747025758269\\
								0.00519366430525251 5.64826906606907\\
								0.00519321648847034 5.64926274115103\\
								0.00519276478127086 5.65066787505583\\
								0.00519224877494546 5.65291613904442\\
								0.00519160769795254 5.65682238894038\\
								0.00519077860912823 5.66387677331709\\
								0.00518969554801417 5.67672614618803\\
								0.00518829210671749 5.69989932994501\\
								0.00518650971964987 5.74074447428797\\
								0.00518431312812319 5.81030297986044\\
								0.00518171183593991 5.92340093364351\\
								0.0051787811184541 6.09681636185659\\
								0.00517566994311733 6.34471267668501\\
								0.0051725819431336 6.67242465310315\\
								0.0051697263775772 7.07236964027257\\
								0.00516725660381201 7.52635366544179\\
								0.00516522692323339 8.01523815739177\\
								0.00516358845182074 8.53296980254347\\
								0.00516221833473338 9.10145602978878\\
								0.00516095897924498 9.78514426774598\\
								0.00515964718207367 10.705131132486\\
								0.005158127002864 12.0492079809947\\
								0.00515625057183171 14.0701019724741\\
								0.00515387311198033 17.0703452390982\\
								0.00515084547778628 21.3899018914894\\
								0.00514700405927156 27.4184004410637\\
								0.00514215762941149 35.635292910574\\
								0.00513607569725605 46.6597373374194\\
								0.00512849143006062 61.2803132019322\\
								0.00511913742571156 80.4296259912022\\
								0.0051078268461492 105.074333113161\\
								0.00509457290229496 136.011679685268\\
								0.00507970941444078 173.597905437746\\
								0.00506394474748038 217.482409294412\\
								0.00504827716015405 266.474962455627\\
								0.00503375450484814 318.688079746364\\
								0.00502116643211154 372.018201913965\\
								0.00501082868845679 424.868691850106\\
								0.00500257118893329 476.87868309656\\
								0.00499590731371817 529.394770801363\\
								0.00499026716740764 585.470745780444\\
								0.0049851866185175 649.226232535192\\
								0.00498040686118266 724.520967188765\\
								0.00497588599501611 813.281942788557\\
								0.0049717405099246 914.283587937844\\
								0.00496814424863756 1023.15717165674\\
								0.00496522649259012 1133.72363613257\\
								0.00496301228177653 1239.99672838353\\
								0.0049614224085809 1337.99640355571\\
								0.00496031484639566 1426.87592859354\\
								0.00495953380029576 1509.40683651922\\
								0.00495894327296009 1592.15592004489\\
								0.00495844099773453 1685.52577815608\\
								0.00495795967051589 1803.28478747341\\
								0.00495746343338765 1960.67210610386\\
								0.00495694367201812 2170.43360031096\\
								0.00495641406436815 2437.76882869879\\
								0.0049559029431188 2756.92474825825\\
								0.0049554423335814 3111.83938520154\\
								0.00495505638606445 3480.9113773242\\
								0.00495475397361786 3843.91513876541\\
								0.00495452830684499 4188.66750474551\\
								0.00495436208711181 4516.14276576702\\
								0.00495423411518361 4844.41739917289\\
								0.0049541240946057 5212.79438708549\\
								0.00495401476490351 5686.5509483657\\
								0.0049538921927532 6359.68223237294\\
								0.00495374544933193 7349.97784745324\\
								0.00495356655004405 8783.2583067941\\
								0.00495335096533373 10772.6721799079\\
								0.00495309840424193 13403.8702060378\\
								0.00495281299193708 16732.7962273442\\
								0.00495250170028978 20801.6316254334\\
								0.00495217043542621 25683.6538867393\\
								0.00495181858629981 31567.4514806202\\
								0.00495143407595693 38876.5189508154\\
								0.00495099067362584 48396.4490699533\\
								0.00495044759265969 61367.5368123089\\
								0.00494974991769249 79521.3859690701\\
								0.00494882848610869 105092.870002254\\
								0.00494759915758821 140863.798212704\\
								0.00494596292155735 190249.16093182\\
							};
							
							\addplot [color=red!50!black, only marks, mark=*, mark options={solid}, forget plot]
							table[row sep=crcr]{
								0.736499267840419 5.02913004821045\\
								0.324129255931391 5.31196150793536\\
								0.0601641012440697 5.5233751201958\\
								0.0096543979965732 5.60755412163217\\
								0.00518431312812319 5.81030297986044\\
								0.00511913742571156 80.4296259912022\\
								0.00498040686118266 724.520967188765\\
								0.00495452830684499 4188.66750474551\\
								0.00495335096533373 10772.6721799079\\
								0.00494974991769249 79521.3859690701\\
							};
							
							\node[below right=3pt, inner sep=0mm, text=black]
							at (axis cs:0.00494974991769249, 79521.3859690701) {1};
							\node[right=3pt, inner sep=0mm, text=black]
							at (axis cs:0.00495335096533373, 10772.6721799079) {2};
							\node[right=3pt, inner sep=0mm, text=black]
							at (axis cs:0.00495452830684499, 4188.66750474551) {3};
							\node[right=3pt, inner sep=0mm, text=black]
							at (axis cs:0.00498040686118266, 724.520967188765) {4};
							\node[right=3pt, inner sep=0mm, text=black]
							at (axis cs:0.00511913742571156, 80.4296259912022) {5};
							\node[below left=3pt, inner sep=0mm, text=black]
							at (axis cs:0.00518431312812319, 5.81030297986044) {6};
							\node[above=3pt, inner sep=0mm, text=black]
							at (axis cs:0.0096543979965732, 5.60755412163217) {7};
							\node[above=3pt, inner sep=0mm, text=black]
							at (axis cs:0.0601641012440697, 5.5233751201958) {8};
							\node[above=3pt, inner sep=0mm, text=black]
							at (axis cs:0.324129255931391, 5.31196150793536) {9};
							\node[above=3pt, inner sep=0mm, text=black]
							at (axis cs:0.736499267840419, 5.02913004821045) {10};		
						\end{axis}
					\end{tikzpicture}%
				}
				\hfill
				\subcaptionbox{\texttt{heat(64)} \label{subfig:heat_64_exp}}[.45\textwidth][c]{
					\begin{tikzpicture}	
						\pgfplotsset{compat=newest,
							label shift=-.01*.35\textwidth,
						}	
						\begin{axis}[%
							width=.35\textwidth,
							scale only axis,
							xmode=log,
							xmin=0.0005,
							xmax=.5,
							xminorticks=true,
							xlabel={$\lVert A x_\alpha - b \rVert_2$},
							ymode=log,
							ymin=1,
							ymax=1000,
							yminorticks=true,
							ylabel={$\lVert x_\alpha \rVert_2$}
							]
							\addplot [
							color=blue,
							solid,
							forget plot
							]
							table[row sep=crcr]{
								0.111409441310603 0.982386261264114\\
								0.093059069683785 1.07233732681117\\
								0.0772665776416634 1.15535305938174\\
								0.0640612929324269 1.23042750601746\\
								0.0533029387682845 1.29736524112059\\
								0.0446963295847756 1.35685141406477\\
								0.0378416590539971 1.41034309151392\\
								0.03231303375275 1.4597781206599\\
								0.0277380875252475 1.50714727936887\\ %
								0.0238477556987663 1.55403382578024\\
								0.0204843083847339 1.60126109261423\\
								0.0175780275189741 1.64876905359518\\
								0.0151108577260903 1.69576040936284\\
								0.0130810942145684 1.74105302112218\\
								0.0114767402467759 1.78350175940068\\ %
								0.010261504087104 1.82234395848519\\
								0.0093748977865048 1.85737854387517\\
								0.00874344688888387 1.88897688373973\\
								0.00829522032358548 1.91799763227843\\
								0.00796999909112223 1.94570488662066\\
								0.00772259154642207 1.97376400003737\\
								0.00752142464803156 2.00433663362418\\
								0.00734543966114189 2.04025274319715\\
								0.00718105865036492 2.08522453924309\\
								0.00701974507522381 2.14408503951262\\
								0.00685612913782712 2.22306592920795\\
								0.00668660839527577 2.33013302913964\\
								0.00650842797208697 2.47531774817005\\
								0.00631922368655066 2.67087337814808\\
								0.006116867182169 2.93112333464927\\
								0.00589942398354664 3.27204505301067\\
								0.00566522903481593 3.71059757198969\\
								0.00541327696337126 4.26340317082079\\
								0.00514402328763536 4.94424558488473\\
								0.0048602682480631 5.76075431853\\ %
								0.00456739562058909 6.71223760856881\\
								0.00427237073020014 7.79115582374944\\
								0.00398169754216234 8.98894741068947\\ %
								0.00369943327893481 10.30375159732\\
								0.0034265249918518 11.7453039311403\\
								0.00316189141288018 13.332829734133\\ 
								0.00290446285953792 15.0856982819019\\
								0.00265480151501105 17.0120373064005\\ %
								0.00241539429591202 19.1027666562775\\
								0.00218979221449353 21.334153967058\\
								0.00198155386197193 23.674277270091\\ %
								0.00179381606841929 26.0855247392989\\
								0.00162946081570833 28.5202421817235\\
								0.0014911409127584 30.9151856624955\\
								0.0013805930631756 33.1937201838212\\
								0.00129754926405569 35.2787380425632\\
								0.00123919891237121 37.1099430711178\\
								0.00120078149271947 38.6558161893224\\
								0.00117691182336679 39.9152940087574\\
								0.00116277752896923 40.9110088502272\\ %
								0.00115471939031207 41.679198941049\\
								0.00115025600100937 42.2605909745514\\
								0.00114783622788471 42.6941865646687\\
								0.00114654485873524 43.0140012599584\\
								0.00114586352666317 43.2479650950089\\
								0.00114550701286408 43.4180952615794\\
								0.00114532157143981 43.5412651999026\\
								0.00114522552588196 43.6301531672586\\
								0.00114517593377016 43.6941531855501\\
								0.00114515038380818 43.7401571886526\\
								0.00114513724120279 43.7731859325775\\
								0.00114513048845004 43.7968786897101\\
								0.00114512702165681 43.8138639245114\\
								0.00114512524287472 43.8260351866816\\
								0.00114512433057633 43.8347540957391\\
								0.00114512386281777 43.8409984875122\\
								0.00114512362303721 43.8454699290886\\
								0.00114512350014059 43.8486714359447\\
								0.0011451234371583 43.850963491406\\
								0.00114512340488354 43.8526043452466\\
								0.00114512338834553 43.8537789618207\\
								0.00114512337987158 43.8546197935134\\
								0.00114512337552971 43.855221677123\\
								0.00114512337330508 43.8556525103207\\
								0.00114512337216527 43.8559609007871\\
								0.00114512337158128 43.856181644971\\
								0.00114512337128208 43.8563396515531\\
								0.00114512337112878 43.8564527507488\\
								0.00114512337105024 43.8565337056526\\
								0.00114512337101 43.8565916521898\\
								0.00114512337098938 43.8566331299709\\
								0.00114512337097881 43.8566628202959\\
								0.0011451233709734 43.8566840745777\\
								0.00114512337097063 43.8566992928712\\
								0.0011451233709692 43.8567101953644\\
								0.00114512337096847 43.8567180177638\\
								0.0011451233709681 43.856723653203\\
								0.0011451233709679 43.8567277578291\\
								0.0011451233709678 43.8567308342046\\
								0.00114512337096774 43.8567333062788\\
								0.0011451233709677 43.856735602862\\
								0.00114512337096766 43.8567382755601\\
								0.00114512337096763 43.8567421967591\\
								0.00114512337096759 43.8567489230042\\
								0.00114512337096753 43.856761387751\\
								0.00114512337096745 43.8567852416843\\
								0.00114512337096735 43.8568314603199\\
								0.0011451233709672 43.8569214274978\\
								0.00114512337096699 43.8570968529861\\
								0.00114512337096671 43.8574391262163\\
								0.0011451233709663 43.858107086904\\
								0.00114512337096574 43.8594107364558\\
								0.00114512337096496 43.8619550756273\\
								0.00114512337096386 43.8669207240053\\
								0.00114512337096233 43.876611161107\\
								0.00114512337096018 43.8955190245078\\
								0.00114512337095719 43.9324004212109\\
								0.00114512337095301 44.0042974180716\\
								0.00114512337094717 44.1442904842669\\
								0.00114512337093901 44.4162614632488\\
								0.00114512337092761 44.9423603892636\\
								0.00114512337091169 45.9518656380709\\
								0.00114512337088944 47.8609439116376\\
								0.00114512337085836 51.3832083030433\\
								0.00114512337081493 57.641324803245\\
								0.00114512337075426 68.2219167867639\\
								0.00114512337066951 85.1674284588299\\
								0.00114512337055109 111.026099664884\\
								0.00114512337038566 149.093927088899\\
								0.00114512337015454 203.853806691252\\ %
								0.00114512336983164 281.567114159198\\
								0.00114512336938053 391.037151345494\\
								0.0011451233687503 544.631064656444\\ %
								0.00114512336786981 759.687824565335\\
								0.00114512336663969 1060.4799485328\\
							};
							
							\addplot [color=red!50!black, only marks, mark=*, mark options={solid}, forget plot]
							table[row sep=crcr]{
								0.0277380875252475 1.50714727936887\\
								0.0114767402467759 1.78350175940068\\
								0.00701974507522381 2.14408503951262\\
								0.0048602682480631 5.76075431853\\
								0.00398169754216234 8.98894741068947 \\
								0.00265480151501105 17.0120373064005\\
								0.00198155386197193 23.674277270091\\
								0.00116277752896923 40.9110088502272\\
								0.00114512337015454 203.853806691252\\
								0.0011451233687503 544.631064656444\\
							};
							
							\node[right=3pt, inner sep=0mm, text=black]
							at (axis cs:0.0011451233687503, 544.631064656444) {1};
							\node[right=3pt, inner sep=0mm, text=black]
							at (axis cs:0.00114512337015454, 203.853806691252) {2};
							\node[right=3pt, inner sep=0mm, text=black]
							at (axis cs:0.00116277752896923, 40.9110088502272) {3};
							\node[right=3pt, inner sep=0mm, text=black]
							at (axis cs:0.00198155386197193, 23.674277270091) {4};
							\node[right=3pt, inner sep=0mm, text=black]
							at (axis cs:0.00265480151501105, 17.0120373064005) {5};
							\node[right=3pt, inner sep=0mm, text=black]
							at (axis cs:0.00398169754216234, 8.98894741068947) {6};
							\node[right=3pt, inner sep=0mm, text=black]
							at (axis cs:0.0048602682480631, 5.76075431853) {7};
							\node[below left=3pt, inner sep=0mm, text=black]
							at (axis cs:0.00701974507522381, 2.14408503951262) {8};
							\node[above=3pt, inner sep=0mm, text=black]
							at (axis cs:0.0114767402467759, 1.78350175940068) {9};
							\node[above=3pt, inner sep=0mm, text=black]
							at (axis cs:0.0277380875252475, 1.50714727936887) {10};
						\end{axis}
					\end{tikzpicture}%
				}
				\caption{Sample values to estimate $\alpha$}
				\label{fig:Lcurv_prob}
			\end{figure}
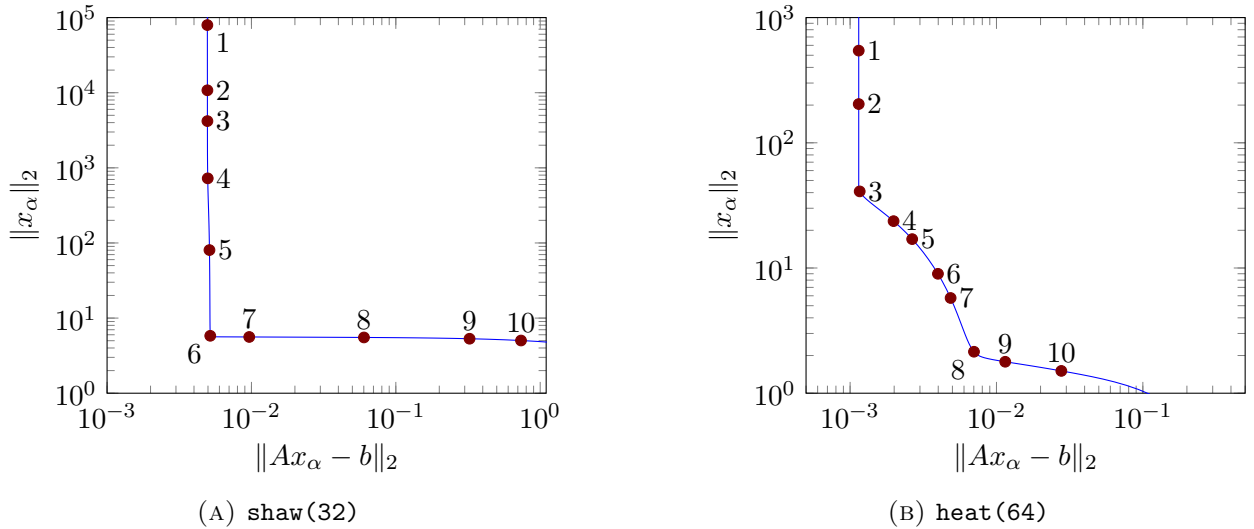

			In the case of the problem \texttt{heat(64)}, considering the points shown in the Table \ref{tab:heat} in $\boldsymbol{\log}$-$\boldsymbol{\log}$ scale and shown in Fig. \ref{subfig:heat_64_exp}, the existence of two points with local maximum of curvature (in the respective neighborhoods of the points labeled $3$ and $8$) can be visually checked.
			\begin{table}[h!]
				\centering
				\caption{Points on the $L$-curve (\texttt{shaw(32)}) and estimated curvature values. The identified corner is highlighted in boldface.}
				\footnotesize
				\begin{tabular}{|l|>{$}l<{$}|>{$}l<{$}|} \hline 
					point index & (\log(\lVert Ax_\alpha - b \rVert_2),\log(\lVert x_\alpha \rVert_2)) & \text{ Estimated curvature values} \\ \thline
					1 & (-5.3084, 11.2838) & 1.1732\times 10^{-5} \\ \hline
					2 & (-5.3077,  9.2848) & 8.2412\times 10^{-5} \\ \hline
					3 & (-5.3075,  8.3401) & 8.7799\times 10^{-4} \\ \hline 
					4 & ( -5.3022,  6.5855) & 3.487\times 10^{-4} \\ \hline
					5 & (-5.2748,  4.3874) & 3.9532\times 10^{-3} \\ \hline  \hline
					\textbf{6} & \mathbf{(-5.2621,  1.7596)} & \mathbf{39.773} \\ \hline \hline
					7 & (-4.6403,  1.7241) & 8.2219 \\ \hline
					8 & (-2.8107,  1.7089) & 1.2114\times 10^{-3} \\ \hline
					9 & (-1.1266,  1.6699) & 7.8662\times 10^{-4} \\ \hline
					10 & (-0.3058,  1.6152) & 1.0444\times 10^{-5} \\ \hline 
				\end{tabular}
				\label{tab:Lcurv}
			\end{table}		
			
			\begin{table}[ht]
				\centering
				\caption{Points on the $L$-curve (\texttt{heat(64)}) and curvature values estimated. The identified corner is highlighted in boldface.}
				\footnotesize
				\begin{tabular}{|l|>{$}l<{$}|>{$}l<{$}|} \hline
					point index & (\log(\lVert Ax_\alpha - b \rVert_2),\log(\lVert x_\alpha \rVert_2)) & \text{Estimated curvature values} \\ \thline
					1 &  (-6.7722, 6.3001) &  2\times 10^{-4} \\ \hline
					2 &  (-6.7722, 4.0542) &  7.56\times 10^{-2} \\ \hline
					3 &  (-6.7569, 3.7114) &  28.51  				\\ \hline
					4 &  (-6.2239, 3.1644) &  4.256\times 10^{-1} \\ \hline
					5 &  (-5.9314, 2.8339) &  1.877\times 10^{-1} \\ \hline
					6 &  (-5.5260, 2.1960) &  1.786\times 10^{-1} \\ \hline
					7 &  (-5.3267, 1.7511) &  3.49\times 10^{-2} \\ \hline \hline
					\textbf{8} &  \mathbf{(-4.9590, 0.7627)} & \mathbf{153.4}	\\ \hline \hline
					9 &  (-4.4674, 0.5786) &  2.812\times 10^{-1} \\ \hline
					10 &  (-3.5849, 0.4102) &  6\times 10^{-3} \\ \hline
				\end{tabular}
				\label{tab:heat}
			\end{table}
			
			For the chosen data, it can be seen in Table \ref{tab:heat} that the local curvature extremes are correctly detected and that point 8 may be considered a global extreme.
			
			As mentioned in Section \ref{sec:preprocessing}, in order to estimate tangent directions on the data (using Algorithm \ref{alg:estima5pts}), each convex sub-polygon must have at least 5 points. In this specific application, it is not necessary to use any of the two variants presented above, which are based on inserting new points (see \cite{Albrecht:convexity}) or tangent directions (see \cite{Estrada-Diaz}) if necessary. Instead, the necessary additional points can be generated on each sub-polygon by varying the regularization parameter and finding few new points on the $L$-curve.

			The authors believe that the curvature estimation method ConicCurv may be advantageously combined with the algorithm {\em discrete L-curve criterion} in \cite{Hansen_pruning}. Let us use the same notations of Adaptive Pruning Algorithm in \cite{Hansen_pruning}. If for each pruned L-curve with $\widehat{p}$ points (assumed to be convex) we compute with the aid of ConicCury estimates of the tangent lines and curvatures at the $\widehat{p}$ points of the pruned L-curve, then we may locate the corner by comparing the curvatures estimated with ConicCurv (observe that the slopes $\phi_j$ in \cite{Hansen_pruning} associated to the candidates of corner points are already computed as intermediate result of applying ConicCury to each each pruned L-curve, i.e., the tangent vectors $\tau_j$ obtained from ABFH method). Certainly, considering the curvature estimates obtained with ConicCurv instead of the angle between subsequent line segments of the pruned L-curve is a more robust and accurate approach, especially in the case of non-uniform arc-length distributed points.
			
			Compared to the algorithmic discrete L-curve criterion, the combination of ConicCurv and the algorithmic discrete L-curve criterion is more computationally expensive, but the computational cost remains in both approaches of the same order, since estimating tangent lines and curvatures at the $\widehat{p}$ points for each pruned L-curve is of order $O(\widehat{p})$.

			\section{Application to subdivision snakes}
			\label{snake}
			
			A very popular technique in curve design and Computer Vision is to construct a spline curve that interpolates a sequence of points on the plane. There are infinitely many smooth curves that satisfy these interpolation conditions, so it is desirable to select the one that \emph{best fits} the data. In the context of image segmentation, a popular approach is to start with an initial user-provided curve configuration that automatically deforms itself to delineate the boundary of the object of interest. The deformation is driven by the minimization of an energy functional \cite{BSUM17}. In this sense, it is introduced the notion of \emph{elastica}, which is the curve that minimizes a certain functional from the Theory of Elasticity \cite{LL59}. This functional depends on the coordinates of the points, and some magnitudes that typically are computed from the curve and its derivatives, such as tangent vectors, arc lengths and curvatures.
			
			The present work may be included in a recent trend for curve and surface subdivision schemes, that consists in taking advantage of the hierarchical nature of subdivision schemes (allowing to generate increasingly dense samples of points on the limit spline curve) to propose algorithms for the computation of energy functionals; see, for instance, \cite{BSUM17,BEBL23,DPhD21,DPHE21}. In this section, we show that it is possible to efficiently compute \emph{elastica} with interpolatory subdivision snakes.
			
			For the sake of simplicity, we sketch how to compute approximations of the functional $E + \lambda\,S$, where $E=\int_0^L \,\kappa^2(s)\,ds $ and $S=\int_0^L \,ds$, for a smooth curve arc length parameterized, with total arc length $L$. This linear combination of the bending energy $E$ and stretching energy $S$ happens to be a popular energy functional in computer vision and image segmentation.

			Given an initial control polygon $\mathcal{P}=\{P_ {i} \in \mathbb{R}^d, i = 1, \ldots, n \}$, we can generate a sequence of polygons by doubling their amount of points in each iteration but preserving the points already defined. In such a setting, a binary interpolatory subdivision scheme $\mathcal{F}: \mathcal{P}^j \to \mathcal{P}^{j+1}$ is a rule that generates recursively polygons $\mathcal{P}^{j+1}=\mathcal{F}(\mathcal{P}^{j})=\{P_ {i}^{j+1}, i = 1, \ldots, 2^{j+1}n\}$ that refine the previous ones. For instance, we can consider a two-point interpolatory subdivision scheme:
			\begin{equation*}
				\begin{cases}
					P_ {2i-1}^{j+1} = P_{i}^{j} \\
					P_ {2i}^{j+1} = F(P_{i-1}^{j} , P_{i}^{j}),
				\end{cases}
			\end{equation*}
			for certain $F:\mathbb{R}^d\times\mathbb{R}^d \to \mathbb{R}^d$ and initial condition $\mathcal{P}^0=\{P_ {i}^0=P_ {i}, i = 1, \ldots, n \}$. For a suitable choice of $F$ the sequence of polygons converges uniformly to a smooth curve, the so-called {\em limit curve} \cite{Sabin}.
			In general, a boundary rule have to be proposed (for $i=2^jn$); however, in the context of our use, we just discard that last point and $\lvert \mathcal{P}^{j+1} \rvert = 2 \lvert \mathcal{P}^{j} \rvert - 1$.

			Without loss of generality, we may restrict the presentation to the approximation of $E$ and $S$ for a convex arc of the limit curve with endpoints $P^0_{i-1}$ and $P^0_{i+1}$. Given $j \geq 1$, iterating $j$-times the subdivision scheme $\mathcal{F}$ with initial polygon $\mathcal{P}^0=\{P_ {i-1}^0,P_ {i}^0 ,P_ {i+1}^0\}$ a refined polygon $\mathcal{P}^{j}=\{P_ {m}^{j}, m = 2^{j}(i-1), \ldots, 2^{j}(i+1) \}$ is obtained.
			
			The $2^{j+1}-1$ chords $\overline{P^j_{m}P^j_{m+2}}$, $m=2^j(i-1), \ldots,2^j(i+1)-2$, have their end points on the arc of the limit curve. Then, the length of the curve section (stretching energy) can be approximated by the sum
			\begin{equation}
				S_j= \sum_{r=0}^{2^j-1} \, \Delta_{2^j(i-1)+2r}^j \quad \text{with} \quad \Delta_m^j \,=\lVert P^j_{m+2} - P^j_{m} \rVert.
				\label{Sj}
			\end{equation}
			On the other hand, the mid point quadrature formula provides an approximation to the bending energy
			
			\begin{equation}
				E_j= \sum_{r=0}^{2^{j-1}-1} \, \kappa\left(P^j_{2^j(i-1)+2r+1}\right)^2 \, \Delta_{2^j(i-1)+2r}^j,
				\label{Ej}
			\end{equation}
			where $\kappa(\cdot)$ denotes the curvature at the specified point.
			
			Recalling that $P^j_m = P^{j+2}_{4m}$, if $\,\,2^{j+2}(i-1)+3 \leq 4m \leq 2^{j+2}(i+1)-3\,\,$ holds, then the set of $\,7\,$ consecutive points $\mathcal{C}^{j}_m=\{P^{j+2}_l, \; l=4m-3, \ldots,4m+3\}$ is contained in $\mathcal{P}^{j+2}$ and we may compute an approximation to the curvature at point $P^j_m$, i.e., $\widetilde{\kappa}(P^j_m)$ by applying Algorithm \ref{alg:AlgCurvature} to $\mathcal{C}^{j}_m$.
			It is straightforward to check that if $m=2^j(i-1)+2r+1$ with $r=0,1, \ldots,2^{j-1}-1$, then it holds $\,\,2^{j+2}(i-1)+3 \leq 4m \leq 2^{j+2}(i+1)-3\,\,$. Hence the bending energy \eqref{Ej} may be approximated by the formula
			\begin{equation*}
				\widetilde{E}_j= \sum_{r=0}^{2^{j-1}-1} \, {(\widetilde{\kappa}(P^{j+2}_{2^{j+2}(i-1)+8r+4}))}^2\, \Delta_{2^j(i-1)+2r}^j,
			\end{equation*}
			where $\widetilde{\kappa}(P^{j+2}_{2^{j+2}(i-1)+8r+4})$ denotes the curvature at $P^{j}_{2^{j}(i-1)+2r+1}=P^{j+2}_{2^{j+2}(i-1)+8r+4}$ estimated by applying Algorithm \ref{alg:AlgCurvature} to $\mathcal{C}^{j}_{2^{j}(i-1)+2r+1}$.
			
			Let $j^*>1$ be a fixed maximum number of iterations of the subdivision scheme $\mathcal{F}$. Without much computational effort, we can achieve estimations of the energy functionals $S$ and $E$ with the accuracy required by Computer Design and Computer Vision applications for relatively small values of $j^*$. For fixed $j^*$ (usually with $j^*$ equals 4 or 5) sufficiently accurate estimations are obtained. For those applications requiring estimates of high order of accuracy, the remarkable effective procedure called Richardson extrapolation may be used; see \cite{FRU07}.
			The pseudocode in Algorithm~\ref{alg:AlgEnergy} shows a very simplified algorithm for a \emph{derivative-free} estimation of the elastic energies \eqref{Sj} and \eqref{Ej}.

			\begin{Algorithm}[H]
				\begin{algorithmic}[1]
					\Require $\, j^*, \,\mathcal{P}^0= \{P_{i-1},P_i,P_{i+1}\}$
					\For{$\, j=1, \ldots,j^{*} \,$ }
					\State $\mathcal{P}^j=\mathcal{F}(\mathcal{P}^{j-1})$;
					\EndFor
					\State $ S=0 $;
					\For{ $r=0, \ldots,2^{j^*}-1$}
					\State $m=2^{j^*}(i-1)+2r$;
					\State $\Delta^{j^*}_m = \lVert P^{j^*}_{m+2}-P^{j^*}_m \rVert$
					\State $S = S+\Delta^{j^*}_m $;
					\EndFor
					\State $E=0$;
					\For{$h=1, \ldots,2^{j^*-1}-1$}
					\State $\displaystyle\delta^{j^*-2}_h=\sum_{r=4(h-1)}^{4h-1} \, \Delta^{j^*}_{2^{j^*}(i-1)+2r}$;
					\State $m=2^{j^*}(i-1)+2h$;
					\State $\kappa^{j^*}_h=\widetilde{\kappa}(P^{j^*}_{m})= \texttt{Algorithm}\,\ref{alg:AlgCurvature}\,(P^{j^*}_{m-3}, \ldots,P^{j^*}_{m+3})$;
					\State $E=E+ {(\kappa^{j^*}_h)}^2\,\delta^{j^*-2}_h$;
					\EndFor
					\Ensure $S\,,\,E$
					\end
					{algorithmic}
					\caption{ Stretch and bending energy estimation with \emph{ConicCurv} method.}
					\label{alg:AlgEnergy}
				\end{Algorithm}

				The computational cost of applying Algorithm \ref{alg:AlgEnergy} with bounded $j^*$ to the subdivision curve arc with initial control polygon $\{P_{i-1},P_i,P_{i+1}\}$ is $O(1)$. Hence, if applied to a control polygon with $n$ vertices, the computational cost is $O(n)$.
				
				The following numerical experiment illustrates the performance of \emph{ConicCurv} to estimate the stretch and bending energy of an ellipse arc by generating increasingly dense samples of points with the binary interpolating subdivision scheme presented in \cite{BEBL23,Diaz-Estrada}. Consider the ellipse from the benchmark of representative curves given in Table \ref{tab:comp}. Let be $P^0_{i-1}=(2.7015,1.6829), \,P^0_{i}=(0.3536,1.994) ,\,P^0_{i+1}=(-2.0807,1.8185)$
				points on the ellipse, the Table \ref{tab:rerror} shows the relative errors of the estimations of stretch and bending energy for the ellipse arc with end points $P^0_{i-1}$ and $P^0_{i+1}$, applying \emph{Algorithm \ref{alg:AlgEnergy}} with $j^*=2,3,4$.
				
				\begin{table}[ht]
					\centering
					\caption{Relative errors of energies $S$ and $E$ estimated with \emph{ConicCurv} for increasing maximum number of iterations $j^*$.}
					\footnotesize
					\begin{tabular}{|c|>{$}c<{$}|>{$}c<{$}|} \hline 
						$j^*$ & \text{ Error of estimated S} & \text{ Error of estimated B} \\ \thline
						2 & 0.0071 & 0.2057 \\ \hline 
						3 & 0.0018 & 0.0665 \\ \hline 
						4 & 0.0004 & 0.0089 \\ \hline 
					\end{tabular}
					\label{tab:rerror}
				\end{table}
				
				\section{Conclusions}
				We have presented \emph {ConicCurv}, a new \emph{derivative-free} and $SE(2)$-invariant algorithm to estimate the curvature of a plane curve from a sample of data points. This algorithm has \emph{conic precision} and is based on a known method to estimate tangent directions, that is grounded on classic results of Projective Geometry and B\'ezier rational conic curves.
				
				Our theoretical study corroborates the results of the presented numerical experiments: in convex settings \emph {ConicCurv}
				has approximation order 3, while classical 3-\emph{points} curvature approximations that are invariant with respect to the special Euclidean group of transformations have approximation order 1. Further, if the sample points are \emph{uniformly arc-length distributed}, the approximation order is 4.

				The performances of \emph{ConicCurv} were shown by comparing it with some of the most frequently used algorithms to estimate curvatures. Finally, its effectiveness was illustrated as a \emph{derivative-free} estimator of the elastic energy of subdivision curves and of the location of L-curve corners.
				An open source implementation is available at \url{https://github.com/RafaelDieF/ConicCurv}.
				
				
				As a completion of the previous researches, a future direction can be pointed out: to extend the curvature estimation method to data on a curve in the space $\mathbb{E}^3$; see \cite{DPhD21,curvat3D}.
				
				\section*{Acknowledgments}
				
				R.D.F. is member of the Gruppo Nazionale per l'Analisi Matematica, la Probabilità  e le loro Applicazioni (GNAMPA) of the Istituto Nazionale di Alta Matematica (INdAM) and acknowledges financial support by PNRR e.INS Ecosystem of Innovation for Next Generation Sardinia (CUP F53C22000430001, codice MUR ECS0000038).

%

			\end{document}